\documentclass[a4paper,12pt,leqno]{amsart}
\usepackage{latexsym}
\usepackage[all]{xy}

\usepackage{amssymb} 
\usepackage{amsmath} 
\usepackage{amscd}
\usepackage{color}
\usepackage{comment}
\usepackage{mathtools}
\usepackage{fancybox}
\usepackage{pb-diagram}
\usepackage{bm}
\usepackage{stmaryrd}
\usepackage{mathrsfs}

\definecolor{gray}{gray}{0.5}

\numberwithin{equation}{section} 

\newtheorem{theorem}{Theorem}[section]
\newtheorem{lemma}[theorem]{Lemma} 
\newtheorem{corollary}[theorem]{Corollary}
\newtheorem{proposition}[theorem]{Proposition} 
 
\newtheorem{remark}[theorem]{Remark}

\newtheorem{definition}[theorem]{Definition}

\newtheorem*{theoremA*}{Theorem A}
\newtheorem*{theoremB*}{Theorem B}

\allowdisplaybreaks[3]

\def\C{\mathbb C}
\def\R{\mathbb R}
\def\Q{\mathbb Q}
\def\Z{\mathbb Z}
\def\P{\mathbb P}

\DeclareMathOperator{\Sym}{Sym}
\DeclareMathOperator{\Hom}{Hom}

\newcommand{\pr}{p}
\newcommand{\LB}{L}
\newcommand{\ctc}[2]{\sigma_{#1,#2}}
\newcommand{\e}[1]{e_{#1}}
\newcommand{\f}[1]{f_{#1}}
\newcommand{\rep}[1]{\dot{#1}}
\newcommand{\cycle}[1]{\llbracket #1\rrbracket}
\newcommand{\PD}[1]{[#1]}
\newcommand{\RL}{\Lambda_{r}}
\newcommand{\WL}{\Lambda}
\newcommand{\CRL}{{\Lambda_{r}^{\vee}}}
\newcommand{\CWL}{\Lambda^{\vee}}
\newcommand{\rkg}{n}

\newcommand{\EL}{E}
\newcommand{\fan}{\Sigma}
\newcommand{\Diva}[1]{D_{-\alpha^{\vee}_{#1}}}
\newcommand{\Divw}[1]{D_{\cvarpi_{#1}}}
\newcommand{\qp}[1]{q_{\alpha_{#1}}}
\newcommand{\nil}{e}
\newcommand{\trs}{\mathsf{T}}
\newcommand{\cvarpi}{\varpi^{\vee}}

\begin{document}

\text{}\vspace{-0pt}
\title[Peterson varieties and toric orbifolds]{Peterson varieties and toric orbifolds\\ associated to Cartan matrices}

\author {Hiraku Abe}
\address{Faculty of Science, Department of Applied Mathematics, Okayama University of Science, 1-1 Ridai-cho, Kita-ku, Okayama, 700-0005, Japan}
\email{hirakuabe@globe.ocn.ne.jp}

\author {Haozhi Zeng}
\address{School of Mathematics and Statistics, Huazhong University of Science and Technology, Wuhan, 430074, P.R. China}
\email{zenghaozhi@icloud.com} 

\begin{abstract}
The Peterson variety is a remarkable variety introduced by Dale Peterson to describe the quantum cohomology rings of all the partial flag varieties. The rational cohomology ring of the Peterson variety is known to be isomorphic to that of a particular toric orbifold which naturally arises from the given root system. 
In this paper, we show that it is not an accidental algebraic coincidence; we construct an explicit morphism from the Peterson variety to the toric orbifold which induces a ring isomorphism between their rational cohomology rings.
\end{abstract}

\maketitle

\section{Introduction}\label{sec: intro}
Let $G$ be a simply connected, semisimple algebraic group over $\C$ and $B$ a Borel subgroup of $G$. To describe the quantum cohomology rings of the partial flag varieties $G/P$ for all parabolic subgroups $P\subseteq G$ containing $B$, Dale Peterson (\cite{Peterson}) introduced a remarkable variety $Y$ in the Langlands dual flag variety $G^{\vee}/B^{\vee}$. After his work, the geometry of the Peterson variety has been studied in several directions (e.g.\ \cite{Ba,In-ty16,Kostant,Peterson,Rietsch06}). By changing the roles of $G$ and $G^{\vee}$, we may define the Peterson variety $Y$ in the flag variety $G/B$ which we take as the convention of this paper.

As for the topology of the Peterson variety, Harada-Horiguchi-Masuda (\cite{ha-ho-ma}) studied the rational cohomology ring $H^*(Y;\Q)$ (cf.\ \cite{AZ,fu-ha-ma,ha-ty11}), and they gave an explicit presentation of $H^*(Y;\Q)$ in terms of ring generators and their relations. Recently, the cohomology ring $H^*(Y;\Q)$ is actively studied from the view point of Schubert calculus (e.g.\ \cite{AHKZ,GoGo,GoMiSi,GoSi,ho23}). 
The ring presentation given by Harada-Horiguchi-Masuda (\cite{ha-ho-ma}) implies that the cohomology ring $H^*(Y;\Q)$ is isomorphic to that of a particular toric orbifold. In fact, there is a toric orbifold $X(\fan)$ introduced by Blume (\cite{Blume05}) to study moduli stacks of pointed chains of $\P^1$, and Harada-Horiguchi-Masuda's presentation implies that
\begin{align}\label{eq: intro 10}
  H^*(Y;\Q) \cong H^*(X(\fan);\Q)
\end{align} 
as graded rings (cf.\ Horiguchi-Masuda-Shareshian-Song \cite[Equation~(1.2)]{ho-ma-sh-so}).
This means that the cohomology ring of the Peterson variety admits a structure of the cohomology ring of a toric orbifold.
We note that it is not difficult to deduce \eqref{eq: intro 10} since explicit presentations for the rings $H^*(Y;\Q)$ and $H^*(X(\fan);\Q)$ are already known from \cite{ha-ho-ma} and \cite[Section~12.4]{CLS}; however the geometric background of \eqref{eq: intro 10} seems still mysterious. 

In Peterson's description of the quantum cohomology rings of the flag varieties, each quantum parameter gives rise to a section of a line bundle over $Y$. Using these sections, 
we construct an explicit morphism 
\begin{align*}
 \Psi \colon Y \rightarrow X(\fan),
\end{align*} 
and we show that $\Psi$ induces a ring isomorphism as in \eqref{eq: intro 10}. 
This means that \eqref{eq: intro 10} is not an accidental algebraic coincidence since the morphism $\Psi$ provides a connection between the geometry of the Peterson variety $Y$ and the geometry of the toric orbifold $X(\fan)$. Indeed, we apply in \cite{AZnn} the construction of $\Psi$ to prove Rietsch's conjecture on the totally nonnegative part of the Peterson variety $Y$ in Lie type A. We emphasize that our aim is not to deduce \eqref{eq: intro 10} but to understand its geometric background.

\vspace{20pt}

\noindent \textbf{Acknowledgments}.
We are grateful to Konstanze Rietsch, Bin Zhang, Changzheng Li, Naoki Fujita, Mikiya Masuda, Takashi Sato, and Tatsuya Horiguchi for valuable discussions. This research is supported in part by Osaka Central Advanced Mathematical Institute (MEXT Joint Usage/Research Center on Mathematics and Theoretical Physics): Geometry and combinatorics of Hessenberg varieties. The first author is supported in part by JSPS Grant-in-Aid for Scientific Research(C): 23K03102. The second author is supported in part by NSFC: 11901218.

\vspace{20pt}

\section{Notations}\label{sec: set up}
Let $G$ be a simply connected, semisimple algebraic group over $\C$ of rank $\rkg$. Choose a Borel subgroup $B\subset G$ and a maximal torus $T\subset B$. Let $W=N(T)/T$ be the Weyl group, where $N(T)$ is the normalizer of $T$ in $G$. Denoting the center of $G$ by $Z$, we have $Z\subset T$.

We denote by $\mathfrak{g}$, $\mathfrak{b}$, and $\mathfrak{t}$ the Lie algebra of $G$, $B$, and $T$, respectively. We have a root space decomposition 
\begin{align*}
 \mathfrak{g}=\mathfrak{t}\oplus\bigoplus_{\alpha\in\Phi}\mathfrak{g}_{\alpha}, 
 \end{align*}
where $\Phi$ is the set of roots. Let $\Phi^+$ be the set of positive roots and $\{\alpha_1,\ldots,\alpha_{\rkg}\}\subset\Phi^+$ the set of simple roots. Denote by $I$ the Dynkin diagram of $\Phi$ which we identify with the indexing set $\{1,\ldots,\rkg\}$ for the simple roots $\alpha_1,\ldots,\alpha_{\rkg}$ unless otherwise specified. 

Let $\RL=\Hom(T/Z,\C^{\times})$ be the root lattice of $T$, and let $\WL=\Hom(T,\C^{\times})$ be the weight lattice of $T$. We have the canonical inclusion $\RL\hookrightarrow \WL$ which takes the pullback by the quotient map $T\rightarrow T/Z$. Regarding them as $\Z$-modules, we have
\begin{align*}
 \RL = \bigoplus_{i\in I} \Z \alpha_i\quad \text{and} \quad \WL = \bigoplus_{i\in I} \Z \varpi_i,
\end{align*}
where $\varpi_1, \ldots, \varpi_{\rkg}$ are the fundamental weights of $T$. 

Dually, let $\CRL=\Hom(\C^{\times},T)$ be the coroot lattice, and let $\CWL=\Hom(\C^{\times},T/Z)$ be the coweight lattice. We also have the canonical inclusion $\CRL\hookrightarrow \CWL$ which is given by the composition with the quotient map $T\rightarrow T/Z$. As $\Z$-modules, we have 
\begin{align*}
 \CRL = \bigoplus_{i\in I} \Z \alpha^{\vee}_i \quad \text{and} \quad \CWL =\bigoplus_{i\in I} \Z \cvarpi_i, 
\end{align*}
where $\alpha^{\vee}_1, \ldots, \alpha^{\vee}_{\rkg}$ are the simple coroots and $\cvarpi_1, \ldots, \cvarpi_{\rkg}$ are the fundamental coweights.
These bases satisfy
\begin{align*}
 \langle \alpha_i , \cvarpi_j \rangle = \delta_{ij}
 \quad \text{and} \quad
 \langle \varpi_i , \alpha^{\vee}_j  \rangle = \delta_{ij}
 \qquad (i,j\in I)
\end{align*}
under the dual parings $\RL\times \CWL \rightarrow \Z$ and $\WL\times \CRL \rightarrow \Z$. Here, we use the same symbol $\langle \ , \ \rangle$ for the both parings by abusing notation.

\vspace{20pt}

\section{Toric orbifolds associated to Cartan matrices}\label{sec: Cartan toric}
In this section, we construct the toric orbifold $X(\fan)$ which appeared in the introduction. We keep the notations from Section~\ref{sec: set up}.

\subsection{The definition by quotient}
We set $\C^{2I}$ as follows:
\begin{align*}
 \C^{2I} \coloneqq 
 \{ (x_1,\ldots,x_{\rkg};y_1,\ldots,y_{\rkg}) \in \C^{2\rkg} \mid x_i,y_i\in\C \ (i\in I) \},
\end{align*}
where $I=\{1,2,\ldots,\rkg\}$ is the Dynkin diagram of the root system $\Phi$. 
The maximal torus $T\subset G$ acts linearly on $\C^{2I}$ through the weights $(\varpi_1,\ldots,\varpi_{\rkg},\alpha_1,\ldots,\alpha_{\rkg})$. 
Namely, we set
\begin{align}\label{eq: def of T-action on C2r}
 t\cdot (x_1,\ldots,x_{\rkg};y_1,\ldots,y_{\rkg}) 
 \coloneqq (\varpi_1(t)x_1,\ldots,\varpi_{\rkg}(t)x_{\rkg}; \alpha_1(t)y_1,\ldots,\alpha_{\rkg}(t)y_{\rkg})
\end{align}
for $t\in T$ and $(x_1,\ldots,x_{\rkg};y_1,\ldots,y_{\rkg})\in \C^{2I}$. 
We define a subset $\EL\subset \C^{2I}$ by
\begin{align}\label{eq: def of exceptional locus}
 \EL \coloneqq \bigcup_{i\in I} \{ (x_1,\ldots,x_{\rkg}; y_1,\ldots,y_{\rkg})\in\C^{2I} \mid  x_i=y_i=0 \}.
\end{align}
Then we have
\begin{align}\label{eq: subtraction}
 \C^{2I}-\EL = \{ (x_1,\ldots,x_{\rkg}; y_1,\ldots,y_{\rkg})\in\C^{2I} \mid  (x_i,y_i)\ne(0,0) \ (i\in I) \}.
\end{align}
It is clear that the linear $T$-action on $\C^{2I}$ defined in \eqref{eq: def of T-action on C2r} preserves the complement $\C^{2I}-\EL$.
We now consider the quotient space 
\begin{align}\label{eq: def of X}
 (\C^{2I}-\EL)/T
\end{align}
by the $T$-action on $\C^{2I}-\EL$ given in \eqref{eq: def of T-action on C2r}.
It admits a natural action of the quotient torus $T/Z$, where $Z$ is the center of $G$. Namely, the torus $T/Z$ acts on $(\C^{2I}-\EL)/T$ by setting
\begin{align}\label{eq: toric action on Xsigma}
 [t]\cdot [x_1,\ldots,x_{\rkg};y_1,\ldots,y_{\rkg}]
 \coloneqq [x_1,\ldots,x_{\rkg}; \alpha_1(t)y_1,\ldots,\alpha_{\rkg}(t)y_{\rkg}]
\end{align}
for $[t]\in T/Z$ and $[x_1,\ldots,x_{\rkg};y_1,\ldots,y_{\rkg}]\in (\C^{2I}-\EL)/T$. This action is well-defined since we have $\alpha_i(t)=1$ $(i\in I)$ for $t\in Z$.

We devote the rest of this section to prove that
the quotient $(\C^{2I}-\EL)/T$ with this $T/Z$-action is a simplicial projective toric variety. 

\begin{remark}
{\rm
As we will see in the next subsection, it coincides with the toric orbifold appeared in the paper of Blume $($\cite{Blume05}$)$ to study moduli stacks of pointed chains of $\P^1$.
In that paper, it was called the toric orbifold associated to the Cartan matrix (see Remark~\ref{rem: Blume}).
}
\end{remark}

\vspace{10pt}

\subsection{A fan $\fan$ on $\mathfrak{t}_{\R}$}\label{sec: construction of fan}
To prove that the quotient $(\C^{2I}-\EL)/T$ (with the action of $T/Z$ described in \eqref{eq: toric action on Xsigma}) is a toric variety, we begin with constructing a fan on $\mathfrak{t}_{\R}\coloneqq \CWL\otimes_{\Z}\R$, where $\CWL=\Hom(\C^{\times},T/Z)=\bigoplus_{i\in I} \Z \cvarpi_i$ is the coweight lattice.
We regard $\mathfrak{t}_{\R}$ as a vector space over $\R$ whose lattice of integral vectors is the coweight lattice $\CWL$.

For disjoint subsets $J,K\subseteq I$ (i.e.\ $J\cap K=\emptyset$), let $\ctc{J}{K}\subset \mathfrak{t}_{\R}$ be the cone spanned by the simple coroots $-\alpha^{\vee}_j$ for $j\in J$ and the fundamental coweights $\cvarpi_k$ for $k\in K$ :
\begin{align}\label{eq: def of cone}
 \ctc{J}{K} \coloneqq \text{cone}( \{-\alpha^{\vee}_j \mid j\in J\}\cup\{\cvarpi_k \mid k\in K \} )\subset \mathfrak{t}_{\R},
\end{align}
where we take the convention $\ctc{\emptyset}{\emptyset}\coloneqq\{\bm{0}\}$. 

\begin{definition}\label{def: fan}
\textnormal{
Let $\fan$ be the set of the cones $\ctc{J}{K}$ for disjoint subsets $J,K\subseteq I$:
\begin{align}\label{eq: def of fan}
 \fan \coloneqq\{\ctc{J}{K} \mid J,K\subseteq I, \ J\cap K=\emptyset\}.
\end{align}
}
\end{definition}

\vspace{10pt}

In what follows, we prove that $\fan$ is a fan on $\mathfrak{t}_{\R}=\CWL\otimes_{\Z}\R$.

\begin{lemma}\label{lem: dim of cones}
For $\ctc{J}{K}\in \fan$, we have
$\dim_{\R} \ctc{J}{K}=|J|+|K|$.
\end{lemma}

\begin{proof}
It suffices to show that the set of ray vectors
\begin{align}\label{eq: generators of the cone}
 \{-\alpha^{\vee}_j \mid j\in J\}\cup\{\cvarpi_k \mid \ k\in K \} 
\end{align}
consists of linearly independent vectors over $\R$.
For that purpose, assume that
\begin{align}\label{eq: linear relation 1}
 \sum_{j\in J} x_j (-\alpha^{\vee}_j) + \sum_{k\in K} x_k \cvarpi_k = \bm{0}
\end{align}
for some coefficients $x_j, x_k\in \R$ $(j\in J, \ k\in K)$.
By taking the paring with $-\alpha_i$ $(i\in J)$, we obtain that
\begin{align*}
 \sum_{j\in J} \langle \alpha_i,\alpha^{\vee}_j \rangle x_j
 - \sum_{k\in K} \langle \alpha_i,\cvarpi_k \rangle x_k
 = 0
 \qquad (i\in J).
\end{align*}
For the second summand in the left hand side, we have $\langle \alpha_i,\cvarpi_k \rangle=\delta_{ik}$. Hence, the condition $J\cap K=\emptyset$ implies that the last equality takes of the form
\begin{align*}
 \sum_{j\in J} \langle \alpha_i,\alpha^{\vee}_j \rangle x_j
 = 0
 \qquad (i\in J).
\end{align*}
By regarding $J(\subseteq I)$ as a full subgraph of the Dynkin diagram $I$, we may think of $J$ as a Dynkin diagram.
Then the coefficient matrix $(\langle \alpha_i,\alpha^{\vee}_j \rangle)_{i,j\in J}$ appearing in this equation is the Cartan matrix associated with the Dynkin diagram $J$.
Hence, it is a nonsingular matrix, and we obtain that $x_j=0$ for $j\in J$.
This and \eqref{eq: linear relation 1} now imply that we also have $x_k=0$ for $k\in K$ since $\cvarpi_k\ (k\in K)$ are linearly independent.
Therefore, the ray vectors in \eqref{eq: generators of the cone} consists of linearly independent vectors.
\end{proof}

\vspace{10pt}

\begin{corollary}\label{cor: scrpc}
Each $\ctc{J}{K}\in \fan$ is a strongly convex rational polyhedral cone in $\mathfrak{t}_{\R}$.
\end{corollary}

\begin{proof}
Since $-\alpha^{\vee}_1,\ldots,-\alpha^{\vee}_{\rkg}$ and $\cvarpi_1,\ldots,\cvarpi_{\rkg}$ are elements of the lattice $\CWL\subset \mathfrak{t}_{\R}$, the cone $\ctc{J}{K}$ defined in \eqref{eq: def of cone} is a rational polyhedral cone.
By Lemma~\ref{lem: dim of cones}, it is generated by linearly independent vectors so that it is strongly convex.
\end{proof}

\vspace{10pt}

The next claim can be proved by a direct computation, but it needs a slightly technical argument. So we prove it in Appendix (Lemma A).

\begin{lemma}\label{lem: intersection of two cones}
For $\ctc{J}{K}, \ctc{P}{Q}\in \fan$, we have $\ctc{J}{K}\cap \ctc{P}{Q} = \ctc{J\cap P}{K\cap Q}$.
\end{lemma}

\vspace{10pt}

\begin{corollary}\label{cor: simplicial fan}
$\fan$ is a simplicial fan on $\mathfrak{t}_{\R}$.
\end{corollary}

\begin{proof}
By Corollary~\ref{cor: scrpc}, $\fan$ is a finite set of strongly convex rational polyhedral cones in $\mathfrak{t}_{\R}$.
As one can see from \eqref{eq: def of cone}, a face of a cone $\ctc{J}{K}\in\fan$ is given by $\ctc{J'}{K'}$ for some $J'\subseteq J$ and $K'\subseteq K$, and hence it belongs to $\fan$ as well.
Also, it follows from Lemma~\ref{lem: intersection of two cones}
that the intersection of two cones in $\fan$ is a face of each.
Thus, $\fan$ is a fan on $\mathfrak{t}_{\R}$.
Moreover, each cone $\ctc{J}{K}$ in $\fan$ is simplicial 
since it is generated by linearly independent vectors as we saw in the proof of Lemma~\ref{lem: dim of cones}. Hence, the claim follows.
\end{proof}

\vspace{10pt}

\begin{definition}
\textnormal{
We denote by $X(\fan)$ the simplicial toric variety 
associated with the fan $\fan$ (on $\mathfrak{t}_{\R}$ with the coweight lattice $\CWL$) defined in \eqref{eq: def of fan}.
}
\end{definition}

\begin{remark}\label{rem: Blume}
{\rm 
In \cite{Blume05}, Blume studied toric stacks whose course moduli spaces are the toric varieties $X(\fan)$ of Lie types A, B, and C.
In that paper, he called them as \textit{toric orbifolds associated to Cartan matrices} since the collection of the ray vectors $-\alpha^{\vee}_1,\ldots,-\alpha^{\vee}_{\rkg}$ is represented by the Cartan matrix of $I$ times $-1$ by taking $\varpi^{\vee}_1, \ldots,\varpi^{\vee}_{\rkg}$ as the basis of the lattice $\CWL$.
}
\end{remark}

\vspace{5pt}

\subsection{The relation between $(\C^{2I}-\EL)/T$ and $X(\fan)$}\label{subsec: quotient construction}
In this subsection, we show that the quotient $(\C^{2I}-\EL)/T$ defined in \eqref{eq: def of X} admits a structure of an algebraic variety which is isomorphic to the toric variety $X(\fan)$.

Let us describe $X(\fan)$ by Cox's quotient construction \cite[Sect.\ 5.1]{CLS}.
Let $\fan(1)$ be the set of rays (i.e.\ 1 dimensional cones) of $\fan$.
Namely,
\begin{align*}
 \fan(1) = \{\text{cone}(-\alpha^{\vee}_i) \mid i\in I\} \cup\{\text{cone}(\cvarpi_i) \mid i\in I\}.
\end{align*}
A collection of rays in $\fan(1)$ is called a primitive collection 
when it is a minimal collection which does not span a cone in $\fan$.
Recalling the definition \eqref{eq: def of cone} of cones $\ctc{J}{K}$ in $\fan$, it is clear that each primitive collection is given by $\{\text{cone}(-\alpha^{\vee}_i), \text{cone}(\cvarpi_i)\}$ for some $i\in I$.
Therefore, the exceptional set $\EL$ in the affine space $\C^{\fan(1)}=\C^{2I}$ determined by these primitive collections coincides with the one defined in \eqref{eq: def of exceptional locus}:
\begin{align*}
 \EL = \bigcup_{i\in I} \{ (x_1,\ldots,x_{\rkg}; y_1,\ldots,y_{\rkg})\in\C^{2I} \mid  x_i=y_i=0 \}
\end{align*}
\vspace{-5pt}\\
(see \cite[Proposition~5.1.6]{CLS}). 
Here, we regard the first $\rkg$ components of $\C^{2I}$ to correspond to the rays $\text{cone}(-\alpha^{\vee}_i)$ in order, and we regard the last $\rkg$ components of $\C^{2I}$ to correspond to the rays $\text{cone}(\cvarpi_i)$ in order.
According to Cox's quotient construction of $X(\fan)$, we consider the complement $\C^{2I}-\EL$, and form a quotient by a certain torus action which we explain below.

Let $\Z^{\fan(1)}=\Z^{2I}$ be the $\Z$-module whose coordinates corresponds to the set of rays in the fan $\fan$ constructed in the previous subsection.
Recall that the lattice of the fan $\fan$ is the coweight lattice $\CWL\subset\mathfrak{t}_{\R}$, and its dual lattice is the root lattice $\RL$.
By definition, $-\alpha^{\vee}_i, \cvarpi_i\in \CWL$ for $i\in I$ are integral ray vectors in $\fan$. 
These ray vectors determine a homomorphism  \vspace{-5pt}
\begin{align}\label{eq: first beta}
 \RL \rightarrow \Z^{2I}
 \quad ; \quad
 \alpha_i \mapsto 
 \sum_{j=1}^{\rkg} \langle\alpha_i,-\alpha^{\vee}_j\rangle \bm{e}_j + \sum_{j=1}^{\rkg} \langle\alpha_i,\cvarpi_j\rangle \bm{e}_{\rkg+j} 
 \quad (i\in I).
\end{align}
\vspace{-5pt}\\
Since we have $\langle\alpha_i,\cvarpi_j\rangle=\delta_{ij}$,
it is clear that this map is injective.
The equalities $\alpha_i = \sum_{j \in I} c_{i,j} \varpi_j$ with $c_{i,j}=\langle\alpha_i,\alpha^{\vee}_j\rangle$ for $i\in I$ imply that the map \eqref{eq: first beta} fits with an exact sequence of character groups of tori
\begin{align}\label{eq: exact sequence of modules}
 0\rightarrow \RL \rightarrow \Z^{2I} \rightarrow \WL \rightarrow 0,
\end{align}
where the third homomorphism sends $\bm{e}_i$ to $\varpi_i$ $(i\in I)$ and $\bm{e}_{\rkg+i}$ to $\alpha_i$ $(i\in I)$.
In fact, one can see the exactness at $\Z^{2I}$ as follows.
The kernel of the third homomorphism consists of the elements 
$x=\sum_{i\in I} a_i \bm{e}_i + \sum_{i\in I} b_i \bm{e}_{\rkg+i} \in \Z^{2I}$
satisfying  $\sum_{i\in I} a_i \varpi_i + \sum_{i\in I} b_i \alpha_i =0$.
Since $\alpha_i=\sum_{j\in I}c_{i,j}\varpi_j$ for $i\in I$, the latter condition is equivalent to $a_i=-\sum_{j\in I} b_jc_{j,i}$ for $j\in I$ which implies that $x$ is the image of $\sum_{i\in I}b_i\alpha_i$ under the map \eqref{eq: first beta}.

Since $\RL$ is the character group of $T/Z$ and $\WL$ is the character group of $T$, the exact sequence \eqref{eq: exact sequence of modules} of the character groups is induced from an exact sequence of tori
\begin{align}\label{eq: exact sequence of tori}
 1\rightarrow T \rightarrow (\C^{\times})^{2I} \rightarrow T/Z \rightarrow 1,
\end{align}
where the second map sends $t\in T$ to $(\varpi_1(t),\ldots,\varpi_{\rkg}(t),\alpha_1(t),\ldots,\alpha_{\rkg}(t))\in (\C^{\times})^{2I}$ since it induces the third homomorphism in \eqref{eq: exact sequence of modules}.
Observe that $(\C^{\times})^{2I}$ acts on $\C^{2I}-\EL$ by the coordinate-wise multiplication which makes $\C^{2I}-\EL$ a toric variety.
Hence, the torus $T$ acts on the complement $\C^{2I}-\EL$ through the injection $T\rightarrow (\C^{\times})^{2I}$ in the exact sequence \eqref{eq: exact sequence of tori}.
By construction, this $T$-action coincides with the action defined in \eqref{eq: def of T-action on C2r}.

The third homomorphism $(\C^{\times})^{2I} \rightarrow T/Z$ in \eqref{eq: exact sequence of tori} extends to a morphism
\begin{align}\label{eq: geometric quotient 1}
 \C^{2I}-\EL\rightarrow X(\fan)
\end{align}
which can be explained as follows. 
Let $\widetilde{\fan}$ be the fan for the toric variety $\C^{2I}-\EL$; it is a fan in $\R^{2I}$ (with the standard lattice $\Z^{2I}$) consisting of all the faces of the cones $\widetilde{\sigma}_{J,K}$ for disjoint subsets $J,K\subseteq I$, where we set 
\begin{align*}
 \widetilde{\sigma}_{J,K} \coloneqq 
 \text{cone}( \{\bm{e}_j \mid j\in J\}\cup\{\bm{e}_{\rkg+k} \mid k\in K \} ) \subseteq \R^{2I}
\end{align*}
(cf.\ \eqref{eq: def of cone}).
The homomorphism $\Z^{2I}\rightarrow \CWL(=\Hom(\C^{\times},T/Z))$ defined by
\begin{align*}
 \bm{e}_i\mapsto -\alpha^{\vee}_i
 \quad \text{and} \quad
 \bm{e}_{\rkg+i}\mapsto \cvarpi_i
 \qquad (i\in I)
\end{align*}
induces a map $\widetilde{\fan}\rightarrow \fan$ between the fans, and the morphism \eqref{eq: geometric quotient 1} is the toric morphism corresponding to this map. In particular, \eqref{eq: geometric quotient 1} is equivariant with respect to the third homomorphism $(\C^{\times})^{2I} \rightarrow T/Z$ in \eqref{eq: exact sequence of tori}.

Recalling that the fan $\fan$ is simplicial (Corollary~\ref{cor: simplicial fan}), we now apply Cox's quotient construction \cite[Theorem~5.1.11]{CLS}.
Here, we note that his original construction requires the chosen ray vectors $-\alpha^{\vee}_i, \cvarpi_i\in \CWL$ $(i\in I)$ to be primitive elements in $\CWL$ whereas 
$-\alpha^{\vee}_i$ may not when the root system $\Phi$ contains that of type A$_1$ or type B$_k$ $(k\ge2)$. However, as claimed in \cite[the proof of Proposition 3.7]{Borisov-Chen-Smith}, the same construction works with possibly non-primitive ray vectors as well (see Remark~\ref{rem: bo-ch-sm} for details).
Namely, it states that the morphism $\C^{2I}-\EL \rightarrow X(\fan)$ in \eqref{eq: geometric quotient 1} is a geometric quotient for the $T$-action on $\C^{2I}-\EL$.
This implies that the morphism \eqref{eq: geometric quotient 1}
induces a bijection 
\begin{align}\label{eq: geometric quotient 2}
 (\C^{2I}-\EL)/T \rightarrow X(\fan)
\end{align}
so that we may regard the quotient $(\C^{2I}-\EL)/T$ as an algebraic variety which is isomorphic to $X(\fan)$ under the map \eqref{eq: geometric quotient 2}.
In particular, we have the following commutative diagram.
\[ \hspace{-20pt}
{\unitlength 0.1in%
\begin{picture}(14.0000,6.0000)(13.0000,-18.7000)%
\put(20.0000,-14.0000){\makebox(0,0)[lb]{$\C^{2I}-\EL$}}%
\put(13.0000,-20.0000){\makebox(0,0)[lb]{$(\C^{2I}-\EL)/T$}}%
\put(26.0000,-20.0000){\makebox(0,0)[lb]{$X(\fan)$}}%
%
\special{pn 8}%
\special{pa 2200 1450}%
\special{pa 1900 1770}%
\special{fp}%
\special{sh 1}%
\special{pa 1900 1770}%
\special{pa 1960 1735}%
\special{pa 1936 1731}%
\special{pa 1931 1708}%
\special{pa 1900 1770}%
\special{fp}%
%
\special{pn 8}%
\special{pa 2400 1450}%
\special{pa 2700 1770}%
\special{fp}%
\special{sh 1}%
\special{pa 2700 1770}%
\special{pa 2669 1708}%
\special{pa 2664 1731}%
\special{pa 2640 1735}%
\special{pa 2700 1770}%
\special{fp}%
\put(23.2000,-19.7000){\makebox(0,0)[lb]{$\stackrel{\cong}{\rightarrow}$}}%
\put(26.5000,-16.0000){\makebox(0,0)[lb]{{\tiny \eqref{eq: geometric quotient 1}}}}%
\put(17.5000,-16.0000){\makebox(0,0)[lb]{{\tiny quot.}}}%
\end{picture}}%
\vspace{7pt}\]
By the equivariance of \eqref{eq: geometric quotient 1} observed above, the isomorphism $(\C^{2I}-\EL)/T \stackrel{\cong}{\rightarrow} X(\fan)$ is equivariant with respect to the group isomorphism of tori $(\C^{\times})^{2I}/T \stackrel{\cong}{\rightarrow} T/Z$ induced by \eqref{eq: exact sequence of tori}.
Recalling that $\C^{2I}-\EL$ is given in \eqref{eq: subtraction}, we obtain that
\begin{align*}
 X(\fan) \cong \{[x_1,\ldots,x_{\rkg};y_1,\ldots,y_{\rkg}] \mid x_i, y_i\in\C, (x_i,y_i)\ne(0,0) \ (i\in I)\}.
\end{align*}

\vspace{5pt}

\begin{remark}\label{rem: bo-ch-sm}
{\rm
Let us briefly explain how \cite[the proof of Proposition 3.7]{Borisov-Chen-Smith} implies that Cox's quotient construction works with possibly non-primitive ray vectors (in our setting).
The collection of ray vectors $\beta\coloneqq\{-\alpha^{\vee}_i \mid i\in I\}\cup \{\cvarpi_i \mid i\in I\}\subseteq \CWL$ defines a surjective homomorphism
\begin{align*}
 \beta \colon \Z^{2I} \rightarrow \CWL
\end{align*}
which sends
$\bm{e}_i$ to $-\alpha^{\vee}_i$ and $\bm{e}_{\rkg+i}$ to $\cvarpi_i$ for $i\in I$.
The triple $(\CWL,\fan,\beta)$ is a stacky fan in the sense of \cite{Borisov-Chen-Smith}, where we note that the coweight lattice $\CWL$ is a free $\Z$-module.
We identify $\Z^{2I}$ and its dual module $\Hom_{\Z}(\Z^{2I},\Z)$ as usual; 
we identify each standard basis vector of $\Z^{2I}$ and the coordinate function for that component.
Then the dual homomorphism
\begin{align*}
 \beta^{\star} \colon \RL \rightarrow \Z^{2I}
\end{align*}
sends each $\alpha_i$ to a vector whose first $\rkg$ components are $(\langle\alpha_i,-\alpha^{\vee}_j\rangle)_{j\in I}$ and the last $\rkg$ components are $(\langle\alpha_i,\cvarpi_j\rangle)_{j\in I}$. Namely, $\beta^{\star}$ is precisely the homomorphism given in \eqref{eq: first beta}.
Therefore, the map $\beta^{\star}$ appeared in the exact sequence \eqref{eq: exact sequence of modules} as its second map: 
\begin{align*}
 0\rightarrow \RL \stackrel{\beta^{\star}}{\rightarrow} \Z^{2I} \rightarrow \WL \rightarrow 0.
\end{align*}
As we explained above, this exact sequence of character groups is induced from the exact sequence \eqref{eq: exact sequence of tori} of tori (equivalently, we obtain \eqref{eq: exact sequence of tori} by taking $\Hom_{\Z}( - ,\C^{\times})$ to this sequence).
We note that the exact sequence \eqref{eq: exact sequence of modules} in our setting is precisely the exact sequence (2.3) in \cite{Borisov-Chen-Smith} so that $\Hom_{\Z}( \WL ,\C^{\times})=T$ is the group to take the quotient in \cite[Sect.~3]{Borisov-Chen-Smith}.
Now, \cite[the proof of Proposition 3.7]{Borisov-Chen-Smith} claims that the morphism $\C^{2I}-\EL\rightarrow X(\fan)$ in \eqref{eq: geometric quotient 1} is the geometric quotient for the $T$-action on $\C^{2I}-\EL$ so that we obtain the conclusions mentioned above.}
\end{remark}

\vspace{10pt}

\subsection{The canonical torus action on $X(\fan)$}
As we saw in the previous subsection, there is an isomorphism $(\C^{2I}-\EL)/T\cong X(\fan)$ which is equivariant with respect to the group isomorphism $(\C^{\times})^{2I}/T \cong T/Z$ of tori in induced by \eqref{eq: exact sequence of tori}.
Thus, the canonical torus of the toric variety $(\C^{2I}-\EL)/T$ can be identified with $T/Z$.

In this subsection, let us describe the action of $T/Z$ on $(\C^{2I}-\EL)/T$ explicitly. For that purpose, we need to write down the above isomorphism $(\C^{\times})^{2I}/T \cong T/Z$ of tori.
Consider a homomorphism
\begin{align*}
 T\rightarrow (\C^{\times})^{2I}/T
 \quad ; \quad
 t\mapsto [1,\ldots,1;\alpha_1(t),\ldots,\alpha_{\rkg}(t)].
\end{align*}
Since we have $\alpha_i(t)=1$ for $t\in Z$ and $i\in I$, this induces a homomorphism 
\begin{align}\label{eq: TZ to C 5}
 T/Z\rightarrow (\C^{\times})^{2I}/T
 \quad ; \quad
 [t]\mapsto [1,\ldots,1;\alpha_1(t),\ldots,\alpha_{\rkg}(t)].
\end{align}

\begin{lemma}\label{lem: inverse homo}
The map \eqref{eq: TZ to C 5} is the inverse of the isomorphism $(\C^{\times})^{2I}/T \rightarrow T/Z$ induced from \eqref{eq: exact sequence of tori}.
\end{lemma}

\begin{proof}
Recall that the exact sequence \eqref{eq: exact sequence of tori} takes of the form
\begin{align}\label{eq: exact sequence of tori''}
 1\rightarrow T \rightarrow (\C^{\times})^{2I} \rightarrow T/Z \rightarrow 1
\end{align}
which was induced by the exact sequence of their character lattices
\begin{align}\label{eq: exact sequence of modules'}
 0\rightarrow \RL \rightarrow \Z^{2I} \rightarrow \WL \rightarrow 0,
\end{align}
where the second map is given by \eqref{eq: first beta}.
Consider a homomorphism
\begin{align}\label{eq: TZ to C}
 T/Z\rightarrow (\C^{\times})^{2I}
 \quad ; \quad
 [t]\mapsto (1,\ldots,1;\alpha_1(t),\ldots,\alpha_{\rkg}(t)).
\end{align}
We first show that this is a splitting of the homomorphism $(\C^{\times})^{2I} \rightarrow T/Z$ in \eqref{eq: exact sequence of tori''}.
By definition, the map \eqref{eq: TZ to C} is induced by the homomorphism of their character lattices
\begin{align*}
 \Z^{2I} \rightarrow \RL
\end{align*}
which sends $\bm{e}_i$ to $0$ and $\bm{e}_{\rkg+i}$ to $\alpha_i$ for $i\in I$.
By composing this with the second map in \eqref{eq: exact sequence of modules'}, we obtain
\begin{align*}
 \RL \rightarrow \Z^{2I} \rightarrow \RL
\end{align*}
which sends
\begin{align*}
 \alpha_i 
 \mapsto 
 \sum_{j=1}^{\rkg} \langle\alpha_i,-\alpha^{\vee}_j\rangle \bm{e}_j + \sum_{j=1}^{\rkg} \langle\alpha_i,\cvarpi_j\rangle \bm{e}_{\rkg+j}
 \mapsto 
 \sum_{j=1}^{\rkg} \langle\alpha_i,\cvarpi_j\rangle \alpha_j
 =\alpha_i
 \qquad (i\in I)
\end{align*}
since we have $\langle\alpha_i,\cvarpi_j\rangle=\delta_{ij}$. 
Namely, this is the identity map.
By construction, this implies that the map $T/Z\rightarrow T/Z$ given by the composition of \eqref{eq: TZ to C} and the homomorphism $(\C^{\times})^{2I} \rightarrow T/Z$ in \eqref{eq: exact sequence of tori''} must be the identity map.
In other words, \eqref{eq: TZ to C} is a splitting of the homomorphism $(\C^{\times})^{2I} \rightarrow T/Z$ in \eqref{eq: exact sequence of tori''}, as we claimed above.

Now observe that the homomorphism \eqref{eq: TZ to C 5} is the composition of \eqref{eq: TZ to C} and the quotient map $(\C^{\times})^{2I}\rightarrow (\C^{\times})^{2I}/T$.
Also, the homomorphism $(\C^{\times})^{2I} \rightarrow T/Z$ in \eqref{eq: exact sequence of tori''} factors though the quotient $(\C^{\times})^{2I}/T$ since the sequence \eqref{eq: exact sequence of tori''} is exact. Therefore, we obtain the following commutative diagram.
\begin{equation*}
\xymatrix{
T/Z \ar[rd]_{\eqref{eq: TZ to C 5}}\ar[r]^{\eqref{eq: TZ to C}} & (\C^{\times})^{2I} \ar[d]\ar[r]^{\eqref{eq: exact sequence of tori''}} & T/Z \\
& (\C^{\times})^{2I}/T \ar[ru]_{\cong} &
}
\end{equation*}
By the claim which we proved above, the composition of the top horizontal maps in this diagram is the identity map. Thus, the commutativity of this diagram implies that the composition $T/Z\stackrel{\eqref{eq: TZ to C 5}}{\rightarrow} (\C^{\times})^{2I}/T \stackrel{\cong}{\rightarrow} T/Z$ of the two slanting maps in this diagram is the identity map as well.
Since the second map in this composition is an isomorphism, it follows that  the map \eqref{eq: TZ to C 5} is its inverse.
\end{proof}

\vspace{10pt}

Lemma~\ref{lem: inverse homo} implies that, under the identification $X(\fan)=(\C^{2I}-\EL)/T$, the canonical torus action of $T/Z$ on $X(\fan)$ is given by
\begin{align}\label{eq: toric action on XSigma}
 [t]\cdot [x_1,\ldots,x_{\rkg};y_1,\ldots,y_{\rkg}] 
 =[x_1,\ldots,x_{\rkg};\alpha_1(t)y_1,\ldots,\alpha_{\rkg}(t)y_{\rkg}] 
\end{align}
for $[t]\in T/Z$ and $[x_1,\ldots,x_{\rkg};y_1,\ldots,y_{\rkg}]\in X(\fan)=(\C^{2I}-\EL)/T$ which is precisely the one given in \eqref{eq: toric action on Xsigma}. Namely, we obtain the following claim.

\begin{proposition}\label{prop: canonical torus action}
The canonical torus of $X(\fan)$ is $T/Z$ which acts on $X(\fan)=(\C^{2I}-\EL)/T$ through the weights $(1,\ldots,1,\alpha_1,\ldots,\alpha_{\rkg})$.
\end{proposition}

\vspace{10pt}

For each $i\in I$, 
we denote by $\Diva{i}$ and $\Divw{i}$ the $T/Z$-invariant irreducible Weil divisors on $X(\fan)$ corresponding to the rays generated by $-\alpha^{\vee}_i$ and $\cvarpi_i$, respectively. 
Under the identification $X(\fan)=(\C^{2I}-\EL)/T$, these divisors  are  given by
\begin{equation}
\begin{split}\label{eq: invariant divisor}
 &\Diva{i} = \{ [x_1,\ldots,x_{\rkg};y_1,\ldots,y_{\rkg}]\in X(\fan) \mid x_i=0 \}, \\
 &\Divw{i} = \{ [x_1,\ldots,x_{\rkg};y_1,\ldots,y_{\rkg}]\in X(\fan) \mid y_i=0 \}.
 \end{split}
\end{equation}
In the rest of this paper, we identify the toric variety $X(\fan)$ and $(\C^{2I}-\EL)/T$ with the $T/Z$-action described in Proposition~\ref{prop: canonical torus action}.

\vspace{20pt}

\section{Projectivity of $X(\fan)$}
In this section, we prove that the variety $X(\fan)$ is projective. We begin with its completeness.

\subsection{Completeness of $X(\fan)$}
For $J\subseteq I$, we set 
\begin{align}\label{eq: def of maximal cone}
 \sigma_J \coloneqq \ctc{J}{I-J} 
 = \text{cone}( \{-\alpha^{\vee}_j \mid j\in J\}\cup\{\cvarpi_k \mid \ k\in I-J \} ).
\end{align}
By Lemma~\ref{lem: dim of cones}, we have 
\begin{align}\label{eq: dim of maximal cone}
 \dim_{\R}\sigma_J = \rkg = \dim_{\R}\mathfrak{t}_{\R}
\end{align}
for $J\subseteq I$. Namely, each $\sigma_J$ is a full-dimensional cone in $\mathfrak{t}_{\R}(=\CWL\otimes_{\Z}\R)$.

\begin{lemma}\label{lem: complete 1}
The maximal cones in $\fan$ are precisely $\sigma_J$ for $J\subseteq I$.
\end{lemma}

\begin{proof}
For a cone $\ctc{J}{K}$ in $\fan$, we have $K\subseteq I-J$ since $J\cap K=\emptyset$.
Hence, we have
\begin{align*}
 \ctc{J}{K} \subseteq \ctc{J}{I-J} = \sigma_J .
\end{align*}
This shows that every cone in $\fan$ is contained in $\sigma_J$ for some $J\subseteq I$.
Moreover, for $J,J'\subseteq I$, we have
\begin{align*}
 \sigma_J \cap \sigma_{J'} = \ctc{J}{I-J}\cap \ctc{J'}{I-J'} = \ctc{J\cap J'}{I-(J\cup J')} 
\end{align*}
by Lemma~\ref{lem: intersection of two cones}.
Hence, we obtain from Lemma~\ref{lem: dim of cones} that
\begin{align*}
 \dim_{\R}(\sigma_J \cap \sigma_{J'}) 
 = |J\cap J'|+|I-(J\cup J')| .
\end{align*}
Thus, if $J\ne J'$, then $\dim_{\R}(\sigma_J \cap \sigma_{J'})<\dim_{\R}\mathfrak{t}_{\R}=\rkg$ which implies that 
either of $\sigma_J$ and $\sigma_{J'}$ is not contained in the other.
\end{proof}

\vspace{10pt}

\begin{lemma}\label{lem: complete 2}
For a codimension $1$ cone $\ctc{J}{K}\in\fan$, there exist $2$ full-dimensional cones in $\fan$ which contains $\ctc{J}{K}$.
\end{lemma}

\begin{proof}
Since $\dim_{\R}\ctc{J}{K}=\rkg-1$ by the assumption, we have $\rkg-(|J|+|K|)=1$ by Lemma~\ref{lem: dim of cones}. So let us write $I-(J\cup K)=\{\ell\}$, and set 
\begin{align*}
 J'\coloneqq J\sqcup\{\ell\}, \quad K'\coloneqq K\sqcup\{\ell\}.
\end{align*}
Then we have $\ctc{J'}{K}=\sigma_{J'}$ and $\ctc{J}{K'}=\sigma_{J}$ by construction. 
Since $J\ne J'$, these are distinct full-dimensional cones containing $\ctc{J}{K}$.
\end{proof}

\vspace{10pt}
We now prove the following.

\begin{proposition}\label{prop: complete}
$X(\fan)$ is complete.
\end{proposition}

\begin{proof}
Let $|\fan|\subseteq\mathfrak{t}_{\R}$ be the support of the fan $\fan$.
By \eqref{eq: dim of maximal cone} and Lemma~\ref{lem: complete 1}, $|\fan|$ is a finite union of the full-dimensional cones in $\fan$.
Thus, the boundary $\partial |\fan|$ is a (possibly empty) union of codimension $1$ cones in $\fan$.
However, Lemma~\ref{lem: complete 2} implies that any codimension 1 cone in $\fan$ does not appear in this union. Therefore, the boundary $\partial |\fan|$ must be empty.
 
Now, suppose that the fan $\fan$ is not complete. Then there exists an element $p\in \mathfrak{t}_{\R}-|\fan|$. 
Choose a cone $\ctc{J}{K}$ in $\fan$ which has the minimum distance from $p$ in $\mathfrak{t}_{\R}$. This cone $\ctc{J}{K}$ is not contained in the interior of $|\fan|$ because of the minimality of the distance. This implies that $\ctc{J}{K}$ and $\partial|\fan|$ have non-empty intersection. In particular, $\partial|\fan|$ must be non-empty (cf.\ \cite[Hint of Exercise~3.4.12]{CLS}). This contradicts to the fact $\partial|\fan|=\emptyset$ proved above.
Therefore, the fan $\fan$ is complete.
\end{proof}

\vspace{10pt}

\subsection{Projectivity of $X(\fan)$}
To show that $X(\fan)$ is projective, we construct an ample divisor on $X(\fan)$ as follows.
For $i\in I$, let $\Diva{i}$ be the torus invariant Weil divisor on $X(\fan)$ corresponding to the ray vector $-\alpha^{\vee}_i$ (see \eqref{eq: invariant divisor}). The goal of this subsection is to prove the following claim.

\begin{proposition}\label{prop: ample divisor}
$\sum_{i\in I}mD_{-\alpha^{\vee}_{i}}$ is an ample divisor on $X(\fan)$ for some $m\in\Z_{>0}$.
In particular, $X(\fan)$ is projective.
\end{proposition}

To prove this claim, we will apply toric Kleiman criterion (\cite[Theorem~6.3.13]{CLS}) since we know that $X(\fan)$ is complete from the previous subsection.
Because of this reason, computations of intersection numbers of divisors and curves in $X(\fan)$ will be important.
Here, since $X(\fan)$ is simplicial, every Weil divisor on $X(\fan)$ is $\Q$-Cartier (\cite[Proposition~4.2.7]{CLS}), and hence intersection numbers of Weil divisors and curves in $X(\fan)$ are defined (e.g.\ \cite[Sect.\ 6.3]{CLS}).

Let us first study the intersection numbers of $\Diva{i}$ and torus invariant irreducible curves in $X(\fan)$.
Such a curve corresponds to a codimension 1 cone $\ctc{J}{K}\in \fan$ (with $J\cap K=\emptyset$) satisfying $|J|+|K|=\rkg-1$ (Lemma~\ref{lem: dim of cones}). 
Let us write $J\sqcup K=I-\{\ell\}$.
To compute intersection numbers of $\Diva{i}$ and this curve, we need to write down a wall relation (with respect to the codimension 1 wall spanned by $\ctc{J}{K}$) as follows.
The cone $\ctc{J}{K}$ is contained in 2 maximal cones $\ctc{\{\ell\}\sqcup J}{K}$ and $\ctc{J}{K\sqcup\{\ell\}}$, and ray vectors of these two cones consist of $\rkg+1$ vectors:
\begin{align*}
-\alpha^{\vee}_{\ell}, \ -\alpha^{\vee}_j\ (j\in J), \ \cvarpi_k\ (k\in K), \ \cvarpi_{\ell} .
\end{align*}
Hence, there is a linear relation (called a wall relation) between those $\rkg+1$ ray vectors which is unique up to a scalar multiplication:
\begin{align}\label{eq: wall relation}
x_{\ell}(-\alpha^{\vee}_{\ell}) + \sum_{j\in J} x_j (-\alpha^{\vee}_j)
+ \sum_{k\in K} y_k \cvarpi_k
+ y_{\ell} \cvarpi_{\ell} = 0
\end{align}
for some coefficients $x_{\ell}>0$, $x_j\in \R$ $(j\in J)$, $y_k\in \R$ $(k\in K)$, $y_{\ell}>0$ (see \cite[Sect.\ 6.3]{CLS}. 
Set $J'\coloneqq \{\ell\}\sqcup J$.
By taking the pairing $\langle -\alpha_i,\ \ \rangle$ to this equation for $i\in J'$, we obtain that
\begin{align*}
x_{\ell} \langle \alpha_i , \alpha^{\vee}_{\ell}\rangle + \sum_{j\in J} x_j \langle \alpha_i , \alpha^{\vee}_{j}\rangle - y_{\ell} \delta_{i\ell}
 = 0 \qquad (i\in J')
\end{align*}
since $J'\cap K=(\{\ell\}\sqcup J)\cap K=\emptyset$.
By including the first summand to the second summand, we can expressed this equality as 
\begin{align*}
\sum_{j\in J'} \langle \alpha_i , \alpha^{\vee}_{j}\rangle x_j 
 = \delta_{i\ell} y_{\ell} 
  \qquad (i\in J').
\end{align*}
Let $C_{J'}=(\langle \alpha_i , \alpha^{\vee}_{j}\rangle)_{i,j\in J'}$ be the Cartan matrix associated to the Dynkin diagram $J'$, as in the proof of Lemma~\ref{lem: dim of cones}.
Then, we obtain
\begin{align*}
x_j = \sum_{i\in J'} (C_{J'}^{-1})_{j,i} \delta_{i\ell}y_{\ell} 
  \qquad (j\in J').
\end{align*}
Since $C_{J'}$ is a Cartan matrix, we have $(C_{J'}^{-1})_{j,i}\ge0$ for all $i,j\in J'$ (\cite{lu-ti92} or \cite[Sect.~13.1]{Humphreys78}).
Recalling that $y_{\ell}>0$ in \eqref{eq: wall relation}, we conclude that $x_j \ge 0$ for $j\in J'$. In particular, we obtain
\begin{align*}
x_j \ge 0 \qquad (j\in J)
\end{align*}
in the wall relation \eqref{eq: wall relation}.
Summarizing this observation, we obtain the following.

\begin{lemma}\label{lem: wall positivity}
For a codimension $1$ cone $\ctc{J}{K}\in\fan$ and $\ell\in I-(J\sqcup K)$, the wall relation \eqref{eq: wall relation} with $x_{\ell}>0$, $y_{\ell}>0$ satisfies $x_j \ge 0$ for all $j\in J$.
\end{lemma}

\vspace{10pt}

Note that $x_j$ appearing in this lemma is the coefficient for the ray vector $-\alpha^{\vee}_j$ in the wall relation \eqref{eq: wall relation}. 
Let $C_{J,K}$ be the torus invariant irreducible curve in $X(\fan)$ corresponding to a codimension 1 cone $\ctc{J}{K}$.
We denote by $\Diva{i}\cdot C_{J,K}\in\Q$ their intersection number.
Now, Lemma~\ref{lem: wall positivity} and \cite[Proposition 6.4.4]{CLS} implies that we have 
\begin{align*}
 \begin{cases}
  \Diva{i}\cdot C_{J,K}
  >0 \quad \text{if $i=\ell$},\\
  \Diva{i}\cdot C_{J,K}
  \ge0 \quad \text{if $i\in J$},\\
  \Diva{i}\cdot C_{J,K}
  =0 \quad \text{if $i\in I-(\{\ell\}\sqcup J)$}.
 \end{cases}
\end{align*}
Thus, we obtain the following claim.

\begin{corollary}\label{cor: non-negative}
For a torus invariant irreducible curve $C\subseteq X(\fan)$ and a torus invariant divisor $\Diva{i}$, the intersection number $\Diva{i}\cdot C$ is non-negative. 
Moreover, for such a curve $C$, there exists $\ell\in I$ such that the intersection number $D_{-\alpha^{\vee}_{\ell}}\cdot C$ is strictly positive.
\end{corollary}

\vspace{10pt}

We now prove Proposition~\ref{prop: ample divisor}.

\begin{proof}[Proof of Proposition~$\ref{prop: ample divisor}$]
Recall that every Weil divisor on $X(\fan)$ is $\Q$-Cartier. Hence,  there exists $m\in\Z_{>0}$ such that $\sum_{i\in I}mD_{-\alpha^{\vee}_{i}}$ is a Cartier divisor.
We show that $\sum_{i\in I}mD_{-\alpha^{\vee}_{i}}$ is ample.
Since $X(\fan)$ is complete (Proposition~\ref{prop: complete}), toric Kleiman criterion \cite[Theorem~6.3.13]{CLS} claims that it is enough to show that the intersection numbers with torus invariant curves are all positive.
The latter criterion follows  from Corollary~\ref{cor: non-negative}, and hence we obtain the desired conclusion.
\end{proof}

\vspace{20pt}

\section{The Peterson variety}
Recall that $B$ is a Borel subgroup of $G$ (see Section~\ref{sec: set up} for the notations).
Let $B^-$ be the opposite Borel subgroup of $G$ so that $T=B\cap B^-$. We denote by $U\subset B$ and $U^-\subset B^-$ the unipotent radicals of $B$ and $B^-$, respectively.
Let $\e{i}\in \mathfrak{g}_{\alpha_i}$ and $\f{i}\in \mathfrak{g}_{-\alpha_i}$ be non-zero elements for each $i\in I$, and let $\nil\in\mathfrak{g}$ be the regular nilpotent element defined by 
\begin{align*}
\nil\coloneqq\sum_{i\in I}\e{i} \in \mathfrak{g}.
\end{align*}
The \textbf{Peterson variety} $Y\subseteq G/B$ is defined to be 
\begin{align}\label{eq:def of Pet}
Y \coloneqq \left\{gB \in G/B \ \left| \ \text{Ad}_{g^{-1}}\nil \in \mathfrak{b} \oplus \bigoplus_{i\in I}\mathfrak{g}_{-\alpha_i}\right. \right\}.
\end{align}
It is known that $Y$ is irreducible (\cite{Peterson} or \cite[Corollary 14]{pre18}).
Since the flag variety $G/B$ is projective, it follows that $Y$ is a projective variety. Note that we have $\dim_{\C}Y=\rkg$, where $\rkg$ is the rank of $G$ (\cite{Peterson} or \cite[Corollary~4.13]{pre13}).

As observed in \cite[Sect.~5]{ha-ty17}, there is a $1$-dimensional torus $S\subseteq T(\subseteq G)$ acting on $Y$ which we now recall in what follows. Consider a surjective homomorphism
\begin{align*}
 \phi \colon T \rightarrow (\C^{\times})^I
 \quad ; \quad
 t\mapsto (\alpha_1(t),\ldots,\alpha_{\rkg}(t))
\end{align*}
between tori of the same dimension.
We define $S\subseteq T$ as a $1$-dimensional subtorus (i.e.\ $S\cong\C^{\times}$) given by
\begin{align*}
 S \coloneqq \text{the identity component of}\ \phi^{-1}\{(c,\ldots,c)\in(\C^{\times})^I\mid c\in\C^{\times}\}.
\end{align*}
By definition, the homomorphisms $\alpha_i|_S \colon S\rightarrow \C^{\times}$ for $i\in I$ are all the same. So we write
\begin{align}\label{eq: def of alpha map}
 \alpha_S(t)\coloneqq\alpha_1(t)=\cdots=\alpha_{\rkg}(t)\in\C^{\times} \qquad \text{for $t\in S$}.
\end{align}
Since $S$ is a subtorus of $T$, it acts on $G/B$ by the multiplication from the left.
The next lemma was observed in \cite[Lemma~5.1]{ha-ty17}.

\begin{lemma}\label{lem: S-action on Y}
The $S$-action on $G/B$ preserves $Y$. 
In particular, $S$ acts on $Y$ by the multiplication from the left.
\end{lemma}

\begin{proof}
For $t\in S$ and $gB\in Y$, we have $tgB\in Y$ since 
\begin{align*}
 \text{Ad}_{(tg)^{-1}}\nil
 =\text{Ad}_{g^{-1}}\text{Ad}_{t^{-1}}\sum_{i\in I}\e{i}
 =\alpha_S(t^{-1})\cdot \text{Ad}_{g^{-1}}\sum_{i\in I}\e{i} \in \mathfrak{b} \oplus \bigoplus_{i\in I}\mathfrak{g}_{-\alpha_i},
\end{align*}
where $\alpha_S(t)$ is the one defined in \eqref{eq: def of alpha map}.
\end{proof}

In \cite[Sect.~5.1]{ha-ty17}, the fixed point set $Y^S$ is explicitly determined as we explain below.
They first observed that the fixed point set of the $S$-action on $G/B$ coincides with that of the $T$-action on $G/B$ (given by the  left multiplication), that is, $(G/B)^S = (G/B)^T$. This implies that we have
\begin{align}\label{eq: S-fixed point 10}
 Y^S = Y\cap (G/B)^S = Y\cap (G/B)^T.
\end{align}
For the fixed point set $(G/B)^T$, there is a bijection 
\begin{align*}
 W \rightarrow (G/B)^T 
 \quad ; \quad w\mapsto \rep{w}B,
\end{align*}
where $\rep{w}\in N(T)$ is a representative of $w\in W=N(T)/T$. 
We identify $W$ and $(G/B)^T$ in the rest of this paper.
For $i\in I$, we denote by $s_i\in W$ the simple reflection associated with $\alpha_i$.
For $J\subseteq I$, we denote by $w_J$ the longest element in the parabolic subgroup $W_J\subseteq W$ generated by $s_j$ for $j\in J$. Then, we have from \cite[Proposition~5.8]{ha-ty17} that
\begin{align}\label{eq: S-fixed point of Y}
 Y^S = \{ w_J \in W \mid J\subseteq I\}.
\end{align}

\vspace{10pt}

\subsection{Line bundles over $G/B$}\label{subsec: line bundles}
As is well-known, a weight of the maximal torus $T$ defines a line bundle over $G/B$, and hence it gives a line bundle over $Y$ by restriction.
To fix notations, let us explain this construction below.
Let $\lambda \colon T\rightarrow \C^{\times}$ be a weight of $T$. By composing this with the canonical projection $B \twoheadrightarrow T$, we obtain a homomorphism $\lambda \colon B\rightarrow \C^{\times}$ which by slight abuse of notation we also denote by $\lambda$. Let $\C_{\lambda}=\C$ be the 1 dimensional representation of $B$ given by $b\cdot z=\lambda(b)z$ for $b\in B$ and $z\in\C$. Since the quotient map $\pr \colon G\rightarrow G/B$ is a principal $B$-bundle, we can take an associated line bundle 
\begin{align*}
 \LB_{\lambda} \coloneqq (G\times \C)/B,
\end{align*}
where $B$ acts on $G\times \C$ by 
\begin{align}\label{eq:def of L_lambda 2}
(g,z)\cdot b=(gb,\lambda(b)z)
\end{align}
for $b\in B$ and $(g,z)\in G\times \C$. 
We note that $\LB_{\lambda}$ admits an action of $G$ given by
\begin{align}\label{eq:def of L_lambda 3}
 g'\cdot [g,z] \coloneqq [g'g,z]
 \qquad (g'\in G, \ [g,z]\in \LB_{\lambda}),
\end{align}
and this makes $\LB_{\lambda}$ a $G$-equivariant line bundle over $G/B$, as is well-known.

\vspace{10pt}

\subsection{A section of $L_{\lambda}$ over $G/B$}\label{subsec: varpi}
In this subsection, we review a construction of a distinguished section of the line bundle $L_{\lambda}$ over $G/B$ associated to a dominant weight $\lambda$. 

Recall that $\WL$ is the weight lattice of $T$.
For a dominant weight $\lambda\in \WL$ (i.e.\ the coefficient of $\lambda$ for each $\varpi_i$ is $\ge0$), we denote by $V(\lambda)$ the irreducible representation of $G$ with the highest weight $\lambda$. 
By regarding it as a representation of $T$, we obtain the weight decomposition:
\begin{align*}
V(\lambda) = \bigoplus_{\mu\in \WL} V(\lambda)_{\mu},
\end{align*}
where $V(\lambda)_{\mu}$ is the weight space of $V(\lambda)$ of weight $\mu\in\WL=\Hom(T,\C^{\times})$. We note that $V(\lambda)_{\mu}=0$ unless $\mu\le \lambda$.
Let $v_{\lambda}\in V(\lambda)$ be a highest weight vector.
We consider a function
\begin{align}\label{eq: def of Delta}
 \Delta_{\lambda} \colon G \rightarrow \C
 \quad ; \quad
 g \mapsto (gv_{\lambda})_{\lambda},
\end{align} 
where $(gv_{\lambda})_{\lambda}$ denotes 
the coefficient of $v_{\lambda}$ for the weight decomposition of $gv_{\lambda}\in V(\lambda)$.
The definition of $\Delta_{\lambda}$ does not depend on the choice of a highest weight vector $v_{\lambda}$.

\begin{remark}\label{rem: picking up coefficient}
{\rm
Let $(\ ,\ )$ be the Shapovalov form on $V(\lambda)$ (see \cite[Sect.\ 3.14 and 3.15]{Humphreys08}) satisfying $(v_{\lambda},v_{\lambda})=1$.
Then one may also express the function $\Delta_{\lambda} \colon G\rightarrow \C$ as $\Delta_{\lambda}(g)=(v_{\lambda},gv_{\lambda})$ for $g\in G$.}
\end{remark}

\vspace{5pt}

\begin{lemma}\label{lem: double equivalence}
For a dominant weight $\lambda\in\WL$, the function $\Delta_{\lambda} \colon G \rightarrow \C_{\lambda}$ is equivariant in the following senses:
\begin{itemize}
\item[$(1)$] $\Delta_{\lambda}(gb) = \lambda(b) \Delta_{\lambda}(g)$ \quad for $g\in G$ and $b\in B$, \vspace{5pt}
\item[$(2)$] $\Delta_{\lambda}(b^-g) = \lambda(b^-) \Delta_{\lambda}(g)$ \quad for $g\in G$ and $b^-\in B^-$.
\end{itemize} 
\end{lemma}

\begin{proof}
We first prove the claim~(1). Recalling that $B=UT$, we can write $b=ut$ for some $u\in U$ and $t\in T$.
Since we have $uv_{\lambda} = v_{\lambda}$ for $u\in U$, it follows that
\begin{align*}
 \Delta_{\lambda}(gb) = (gutv_{\lambda})_{\lambda}
 = \lambda(t)(gv_{\lambda})_{\lambda}.
\end{align*}
Here, we have $\lambda(t)=\lambda(b)$ since we wrote $b=ut$ above.
Therefore, we obtain the desired equality $\Delta_{\lambda}(gb) = \lambda(b) \Delta_{\lambda}(g)$.

To prove the claim~(2), we take the decomposition $B^-=U^-T$ so that we can write $b^-=u^-t$ for  some  $u^-\in U^-$ and $t\in T$.
Observe that $\lambda(t)gv_{\lambda}$ and $u^-tgv_{\lambda}$ has the same highest weight component for their weight decompositions.
Thus, we obtain the desired claim as in the previous case.
\end{proof}

\vspace{10pt}

We now obtain the following corollary which characterizes the function $\Delta_{\lambda}$.

\begin{corollary}
For a dominant weight $\lambda\in\WL$, the function $\Delta_{\lambda} \colon G \rightarrow \C$ is a regular function on $G$ whose restriction to the open dense subset $U^-B\subseteq G$ is given by
\begin{align*}
 \Delta_{\lambda}(u^-b) = \lambda(b) 
 \qquad (u^-\in U^-,\ b\in B).
\end{align*} 
In particular, $\Delta_{\lambda}$ is an extension of the character $\lambda \colon B \rightarrow \C^{\times}$.
\end{corollary}

\vspace{10pt}

Because of the equivariance given in Lemma~\ref{lem: double equivalence}~(1), the function $\Delta_{\lambda} \colon G\rightarrow \C$ (for a dominant weight $\lambda$) gives rise to a well-defined section of $\LB_{\lambda}$ :
\begin{align}\label{eq: psi i by ()}
 \psi_{\lambda} \colon G/B \rightarrow \LB_{\lambda} 
 \quad ; \quad
 gB \mapsto [g,\Delta_{\lambda}(g)] ,
\end{align}
where the value $\Delta_{\lambda}(g)=(gv_{\lambda})_{\lambda}$ is the one defined in \eqref{eq: def of Delta}.
As we will see in Section~\ref{subsec: zeros of divisors}, the case $\lambda=\varpi_i$ $(i\in I)$ will play an important role in this paper.

\vspace{10pt}

\subsection{A distinguished section of $L_{\alpha_i}$ over $Y$}\label{subsec: alphai}
In general, each simple root $\alpha_i$ is not a dominant weight.
This means that the line bundle $L_{\alpha_i}$ over $G/B$ is not effective so that it does not admit a non-trivial global section.
However, after restricting on $Y$, the line bundle $L_{\alpha_i}$ admits non-trivial global sections over $Y$ as we now see below.

Let $P_Y$ be the restriction of the principal $B$-bundle $G\rightarrow G/B$ over the Peterson variety $Y\subseteq G/B$. By the definition \eqref{eq:def of Pet} of $Y$, it is simply given by 
\begin{align}\label{eq: def of PY}
 P_Y 
 = \{g\in G \mid gB\in Y\} 
 =\left\{g \in G \ \left| \ \text{Ad}_{g^{-1}}\nil \in \mathfrak{b} \oplus \bigoplus_{i\in I}\mathfrak{g}_{-\alpha_i}\right. \right\}.
\end{align}
By definition, $P_Y\subseteq G$ is preserved by the multiplication of elements of $B$ from the right, and hence $P_Y$ inherits a right $B$-action. With this action, $P_Y$ is a principal $B$-bundle over $Y$ in the 
sense that the projection map $P_Y\rightarrow Y$ is $B$-equivariantly locally trivial.
We note that $P_Y$ is irreducible since the fiber and the base space of the projection map $P_Y\rightarrow Y$ are both irreducible (cf.\ \cite[the proof of Theorem1.26]{Shafarevich}).
We now have the pullback diagram of principal $B$-bundles.
\begin{align}\label{eq: diagram of PY}
\begin{matrix}
\xymatrix{
P_Y \ar[d]\ar[r] & G \ar[d] \\
Y \ar[r] & G/B
}
\end{matrix}
\end{align}
By the same argument for the proof of Lemma~\ref{lem: S-action on Y}, 
it is straightforward to see that $P_Y$ admits an $S$-action given by the multiplication from the left.
Therefore, $P_Y$ has actions of $S$ and $B$ by the multiplication from the left and the right, respectively.
For $i\in I$, consider a function
\begin{align}\label{eq: function for small Pet}
 \qp{i} \colon P_Y \rightarrow \C
 \quad ; \quad
 g \mapsto -(\text{Ad}_{g^{-1}}\nil)_{-\alpha_i}, 
\end{align}
where $(\text{Ad}_{g^{-1}}\nil)_{-\alpha_i}$ denotes the coefficient of the root vector $\f{i}\in\mathfrak{g}_{-\alpha_i}$ for the root decomposition of $\text{Ad}_{g^{-1}}\nil\in\mathfrak{g}$.

\begin{lemma}\label{lem: single equivalence}
For $i\in I$, the function $\qp{i} \colon P_Y \rightarrow \C$ is equivariant in the following sense:
\begin{itemize}
\item[$(1)$] 
$\qp{i}(gb) = \alpha_i(b)\qp{i}(g)$
 \quad for $g\in P_Y$ and $b\in B$, \vspace{5pt}
\item[$(2)$] 
$\qp{i}(tg) =\alpha_i(t^{-1})\qp{i}(g)$
 \quad for $g\in P_Y$ and $t\in S$,
\end{itemize} 
where $\alpha_S(t)$ is defined in \eqref{eq: def of alpha map}.
\end{lemma}

\begin{proof}
We first prove the claim~(1).
Since $g\in P_Y$, we have $\text{Ad}_{g^{-1}}\nil \in \mathfrak{b} \oplus \bigoplus_{i\in I}\mathfrak{g}_{-\alpha_i}$ by \eqref{eq: def of PY} so that we can write
\begin{align*}
\text{Ad}_{g^{-1}}\nil =x + \sum_{i\in I} c_{i}\f{i}
\end{align*}
for some $x\in\mathfrak{b}$ and $c_{i}\in\C$ for $i\in I$. 
This means that $\qp{i}(g)=-c_{i}$ by definition.
Recalling that $B=UT$, we can write $b=ut$ for some $u\in U$ and $t\in T$.
We have $\text{Ad}_u \f{i}\in \f{i}+\mathfrak{b}$, and hence
\begin{align*}
 \text{Ad}_{(gb)^{-1}}\nil =\text{Ad}_{t^{-1}}\text{Ad}_{u^{-1}}\cdot\text{Ad}_{g^{-1}}\nil =x' + \sum_{i\in I} \alpha_i(t)c_{i}\f{i} 
\end{align*}\vspace{-10pt}\\
for some $x'\in\mathfrak{b}$.
This means that
\begin{align*}
 \qp{i}(gb) =
 -(\text{Ad}_{(gb)^{-1}}\nil)_{-\alpha_i} = -\alpha_i(t)c_{i} =\alpha_i(t) \qp{i}(g).
\end{align*}
Since we wrote $b=ut$ above, we have $\alpha_i(b)=\alpha_i(t)$.
Hence, we obtain the claim~(1).

For the claim~(2), we have 
\begin{align*}
 \text{Ad}_{(tg)^{-1}}\nil 
 =\text{Ad}_{g^{-1}}\text{Ad}_{t^{-1}}\sum_{i\in I} \e{i}
 =\alpha_S(t^{-1})\text{Ad}_{g^{-1}}\sum_{i\in I} \e{i}
 =\alpha_i(t^{-1})\text{Ad}_{g^{-1}}\nil
\end{align*}
by the definition of $\alpha_S(t)$ in \eqref{eq: def of alpha map}.
Thus, the claim~(2) follows.
\end{proof}

\vspace{10pt}

Because of the equivariance given in Lemma~\ref{lem: single equivalence}~(1), the function $\qp{i} \colon P_Y \rightarrow \C$ gives rise to a well-defined section of $L_{\alpha_i}|_Y$ over $Y$ :
\begin{align}\label{eq: desired section 1}
 \phi_{\alpha_i} \colon Y \rightarrow \LB_{\alpha_i}|_Y
 \quad ;\quad
 gB \mapsto [g,\qp{i}(g)] ,
\end{align}
where the value $\qp{i}(g)=-(\text{Ad}_{g^{-1}}\nil)_{-\alpha_i}$ is the one defined in \eqref{eq: function for small Pet}.

\vspace{20pt}

\subsection{Intersections of divisors of zeros on $Y$}\label{subsec: zeros of divisors}
The goal of this subsection is to prove Proposition~\ref{prop: non-simultaneous goal} below.
This claim will be important to construct a morphism $Y\rightarrow X(\fan)=(\C^{2I}-\EL)/T$ because of \eqref{eq: subtraction}.

\begin{proposition}\label{prop: non-simultaneous goal}
Let $i\in I$. Then we have
\begin{align*}
 \{gB \in Y \mid \Delta_{\varpi_i}(g)=0, \ \qp{i}(g)=0\}=\emptyset.
\end{align*}
Namely, $\Delta_{\varpi_i}(g)$ and $\qp{i}(g)$ cannot be zero simultaneously for each $gB\in Y$.
\end{proposition}

\vspace{10pt}

To prove this, we consider the zero loci of the sections $\psi_{\varpi_i}\colon G/B\rightarrow \LB_{\varpi_i}$ and $\phi_{\alpha_i}\colon Y\rightarrow \LB_{\alpha_i}|_Y$ since these are determined by the functions $\Delta_{\varpi_i}$ and  $\qp{i}$, respectively.

We begin with the zero locus of $\psi_{\varpi_i}$.
Since we have $\psi_{\varpi_i}(gB)=[g,\Delta_{\varpi_i}(g)]$ for $gB\in G/B$ by definition, 
Lemma~\ref{lem: double equivalence}~(2) implies that the zero locus $Z(\psi_{\varpi_i})\subset G/B$ is a Cartier divisor which is preserved by the $B^-$-action on $G/B$ from the left. Hence, it is a union of opposite Schubert divisors $\Omega_{s_k}=\overline{B^-\rep{s}_k B/B}$ $(k\in I)$, where $\rep{s}_k\in N(T)$ is a representative of the simple reflection $s_k\in W$.
The following well-known claim states that this union consists of a single divisor.

\begin{lemma}\label{lem: zero=Schubert divisor}
For $i\in I$, we have $Z(\psi_{\varpi_i})=\Omega_{s_i}$ in $G/B$.
\end{lemma}

\begin{proof}
As we observed above, $Z(\psi_{\varpi_i})$ is a union of opposite Schubert divisors.
We can determine this union by verifying which $\rep{s}_k B$ $(k\in I)$ belongs to $Z(\psi_{\varpi_i})$ since $\rep{s}_k B$ belongs to precisely one opposite Schubert divisor $\Omega_{s_k}$.
Thus, it suffices to prove that
\begin{align}\label{eq: sj in Omega si}
 \psi_{\varpi_i}(\rep{s}_k B)=0 \quad \text{if and only if\quad$k=i$}.
\end{align}
So let us prove \eqref{eq: sj in Omega si} in what follows.
By the definition \eqref{eq: psi i by ()}, 
the value $\psi_{\varpi_i}(\rep{s}_k B)$ reads the coefficient of the highest weight vector $v_{\varpi_i}$ for the weight decomposition of $\rep{s}_k v_{\varpi_i}\in V(\varpi_i)$. 
This vector $\rep{s}_k v_{\varpi_i}$ is a weight vector of weight $s_k(\varpi_i)$ (\cite[Sect.\ 31.1]{Humphreys75}) which can be computed as 
\begin{align*}
 s_k(\varpi_i)
 =\varpi_i-\langle \varpi_i,\alpha^{\vee}_k\rangle\alpha_k 
 =
 \begin{cases}
  \varpi_i-\alpha_i \quad &\text{if $k=i$}, \\
  \varpi_i &\text{if $k\ne i$}.
 \end{cases}
\end{align*}
Therefore, the highest weight component of $\rep{s}_k v_{\varpi_i}$ is zero if and only if $k=i$.
That is, 
\begin{align*}
 \psi_{\varpi_i}(\rep{s}_k B)=0
 \quad \text{if and only if} \quad
 k=i
\end{align*}
which is precisely the desired claim \eqref{eq: sj in Omega si}.
\end{proof}

\vspace{10pt}

Recalling that $\phi_{\alpha_i}(gB)=[g,q_{\alpha_i}(g)]=[g,-(\text{Ad}_{g^{-1}}\nil)_{-\alpha_i}]$ for $gB\in G/B$, we have
\begin{equation*}
\begin{split}
 Z(\phi_{\alpha_i}) 
 = \{gB \in Y \mid \qp{i}(g)=0\} 
 = \{gB \in Y \mid (\text{Ad}_{g^{-1}}\nil)_{-\alpha_i}=0\},
\end{split}
\end{equation*}
By Lemma~\ref{lem: single equivalence} (2), this subset $Z(\phi_{\alpha_i})\subseteq Y$ is preserved by the $S$-action on $Y$.
The next claim is a special case of \cite[Proposition~5.2]{ha-ty17}, but we give a direct proof for the readers convenience.

\begin{lemma}\label{lem: fixed points of small Peterson}
For $i\in I$, we have 
\begin{align*} 
 Z(\phi_{\alpha_i})^S = \{ w_J \in W \mid i\notin J\subseteq I \}
\end{align*}
under the description of $Y^S$ given in \eqref{eq: S-fixed point of Y}.
\end{lemma}

\begin{proof}
Since we have
\begin{align*}
 Z(\phi_{\alpha_i})^S\subseteq Y^S= \{w_J\in W\mid J\subseteq I\}
\end{align*}
by \eqref{eq: S-fixed point of Y}, it suffices to show that $\phi_{\alpha_i}(\rep{w}_J B)=0$ if and only if $i\notin J$, where $\rep{w}_J\in N(T)$ is a representative of $w_J$.
That is, we need to show that $q_{\alpha_i}(\rep{w}_J)=(\text{Ad}_{\rep{w}_J^{-1}}e)_{-\alpha_i}=0$ if and only if $i\notin J$.
To this end, recall that
\begin{align}\label{eq: property of wJ 3}
 \text{Ad}_{\rep{w}_J^{-1}}e 
 = \sum_{k\in I}\text{Ad}_{\rep{w}_J^{-1}}e_k .
\end{align}
In the right hand side, each $\text{Ad}_{\rep{w}_J^{-1}}e_k$ is a root vector for the root $w_J^{-1}\alpha_k$.
This means that $q_{\alpha_i}(\rep{w}_J)=-(\text{Ad}_{\rep{w}_J^{-1}}e)_{-\alpha_i}=0$ if and only if 
$w_J^{-1}\alpha_k\ne -\alpha_i$ for all $k\in I$ in \eqref{eq: property of wJ 3}.
Note that we have
\begin{align}
\begin{split}\label{eq: property of wJ}
 w_J^{-1} (\Delta_J) = -\Delta_J \quad \text{and} \quad 
 w_J^{-1} (\Delta - \Delta_J) \subseteq \Phi^+ ,
\end{split}
\end{align}
where $\Delta$ is the set of simple roots and $\Delta_J\coloneqq \{\alpha_i\mid i\in J\}$
(e.g.\ the first paragraph of the proof of \cite[Lemma~5.7]{ha-ty17}).
So let us take the following decomposition of \eqref{eq: property of wJ 3}:
\begin{align}\label{eq: property of wJ 2}
 \text{Ad}_{\rep{w}_J^{-1}}e 
 = \sum_{j\in J}\text{Ad}_{\rep{w}_J^{-1}}e_j + \sum_{k\in I-J}\text{Ad}_{\rep{w}_J^{-1}}e_k .
\end{align}
By \eqref{eq: property of wJ}, we have $w_J^{-1}\alpha_k\ne -\alpha_i$ for all $k\in I$ in \eqref{eq: property of wJ 2} if and only if $i\notin J$.
Therefore, we conclude that $q_{\alpha_i}(\rep{w}_J)=0$ if and only if $i\notin J$, as desired.
\end{proof}

\vspace{10pt}

In the next claim, we denote the restricted section $\psi_{\varpi_i}|_Y\colon Y\rightarrow L_{\varpi_i}|_Y$ by the same symbol $\psi_{\varpi_i}$.

\begin{proposition}\label{prop: empty intersection}
For $i\in I$, we have 
\begin{align*}
 Z(\psi_{\varpi_i}) \cap Z(\phi_{\alpha_i})= \emptyset \quad \text{in $Y$},
\end{align*}
where $\psi_{\varpi_i}\colon Y\rightarrow L_{\varpi_i}|_Y$ and $\phi_{\alpha_i}\colon Y\rightarrow L_{\alpha_i}|_Y$ are the sections constructed in Section~$\ref{subsec: varpi}$ and Section~$\ref{subsec: alphai}$, respectively.
\end{proposition}

\begin{proof}
Since $S\subseteq T\subseteq B^-$, Lemma~\ref{lem: double equivalence}~(2) and Lemma~\ref{lem: single equivalence}~(2) imply that the zero loci $Z(\psi_{\varpi_i})$ and $Z(\phi_{\alpha_i})$ are both closed subsets of $Y$ preserved under the $S$-action on $Y$.
We first show that 
\begin{align}\label{eq: fixed point intersection empty}
 Z(\psi_{\varpi_i})^S\cap Z(\phi_{\alpha_i})^S=\emptyset
\end{align}
in the following.
Since we have $Z(\psi_{\varpi_i})=Y\cap \Omega_{s_i}$ by Lemma~\ref{lem: zero=Schubert divisor}, it follows that
\begin{align*}
 Z(\psi_{\varpi_i})^S = Y^S\cap \Omega_{s_i}^S = \{w_J \in W \mid J\subseteq I, \ w_J\ge s_i\} ,
\end{align*}
where we used the fact $\Omega_{s_i}^S = \Omega_{s_i}^T$ which follows from $(G/B)^S = (G/B)^T$.
By the definition of the Bruhat order, this implies that
\begin{align}\label{eq: fixed point 1}
 Z(\psi_{\varpi_i})^S = \{w_J \in W \mid i\in J\subseteq I\}.
\end{align}
By Lemma~\ref{lem: fixed points of small Peterson}, we also have
\begin{align}\label{eq: fixed point 2}
 Z(\phi_{\alpha_i})^S
 = \{w_J \in W \mid i\notin J\subseteq I\}.
\end{align}
Therefore, we obtain \eqref{eq: fixed point intersection empty}, as claimed above.

We now assume that $Z(\psi_{\varpi_i}) \cap Z(\phi_{\alpha_i})\ne \emptyset$, and we deduced a contradiction. This assumption means that the intersection $Z(\psi_{\varpi_i}) \cap Z(\phi_{\alpha_i})$ is a non-empty closed subset of $Y$ preserved under the $S$-action. Since $S(\cong \C^{\times})$ is a torus, it follows that it must contain an $S$-fixed point (e.g.\ \cite[Sect.\ 21.2]{Humphreys75}). Hence, we have 
\begin{align*}
 Z(\psi_{\varpi_i})^S\cap Z(\phi_{\alpha_i})^S  = (Z(\psi_{\varpi_i}) \cap Z(\phi_{\alpha_i}))^S \ne\emptyset
\end{align*}
which contradicts to \eqref{eq: fixed point intersection empty}.
\end{proof}

\vspace{10pt}
Recalling \eqref{eq: psi i by ()} and \eqref{eq: desired section 1}, it is clear that Proposition~\ref{prop: empty intersection} implies Proposition~\ref{prop: non-simultaneous goal} stated in the beginning of this subsection.

\vspace{20pt}

\section{A morphism from $Y$ to $X(\fan)$}\label{sect: the morphism}
In this section, we construct a morphism $\Psi \colon Y \rightarrow X(\fan)$ from the Peterson variety $Y$ to the toric orbifold $X(\fan)=(\C^{2I}-\EL)/T$, and we prove that it induces a ring isomorphism $\Psi^* \colon H^*(X(\fan);\Q) \rightarrow H^*(Y;\Q)$ on the cohomology rings with $\Q$ coefficients.
After that, we describe the correspondence on the degree $2$ parts explicitly, and we explain the compatibility of $\Psi^*$ with the known presentations for the rings $H^*(X(\fan);\Q)$ and $H^*(Y;\Q)$.

\subsection{Construction of the morphism $\Psi$}
Recall from \eqref{eq: diagram of PY} that $P_Y\subseteq G$ is the restriction of the principal $B$-bundle $G\rightarrow G/B$ over the Peterson variety $Y$.
From Sections~\ref{subsec: varpi} and \ref{subsec: alphai}, we have two kinds of functions 
\begin{align*}
 &\Delta_{\varpi_i} \colon P_Y \rightarrow \C
 \quad ; \quad
 g \mapsto (gv_{\varpi_i})_{\varpi_i},\\
 &\hspace{6.2pt}\qp{i} \colon P_Y \rightarrow \C
 \quad ; \quad
 g \mapsto -(\text{Ad}_{g^{-1}}\nil)_{-\alpha_i},
\end{align*}
for $i\in I$ which produces the sections $\psi_{\varpi_i}$ and $\phi_{\alpha_i}$ of the line bundles $\LB_{\varpi_i}|_Y$ and $\LB_{\alpha_i}|_Y$, respectively.
Using these functions, we obtain a morphism $\widetilde{\Psi} \colon  P_Y \rightarrow \C^{2I}$ given by
\begin{align*}
 \widetilde{\Psi}(g)= (\Delta_{\varpi_1}(g),\ldots,\Delta_{\varpi_{\rkg}}(g); \qp{1}(g), \ldots, \qp{\rkg}(g))
\end{align*}
for $g\in P_Y$, where we write $I=\{1,2,\ldots,\rkg\}$.
Since $g\in P_Y$ implies that $gB\in Y$, 
Proposition~\ref{prop: non-simultaneous goal} implies that $\Delta_{\varpi_i}(g)$ and $\qp{i}(g)$ in the right hand side cannot be zero simultaneously for each $i\in I$ so that the image of $\widetilde{\Psi}$ lies in $\C^{2I}-\EL$ by \eqref{eq: subtraction}. Namely, we have 
\begin{align}\label{eq: Psi tilde}
 \widetilde{\Psi} \colon P_Y \rightarrow \C^{2I}-\EL.
\end{align}
Recall that $P_Y$ admits the multiplication of $B$ from the right, and $\C^{2I}-\EL$ admits the $T$-action defined in \eqref{eq: def of T-action on C2r}.
With these in mind, Lemma~\ref{lem: double equivalence}~(1) and Lemma~\ref{lem: single equivalence}~(1) now imply that the map \eqref{eq: Psi tilde} is equivariant with respect to the projection homomorphism $B\twoheadrightarrow T$.
Composing $\widetilde{\Psi}$ with the quotient morphism $\C^{2I}-\EL\rightarrow (\C^{2I}-\EL)/T=X(\fan)$, we thus obtain a morphism $P_Y \rightarrow X(\fan)$ which is $B$-invariant. Since the projection map $P_Y\rightarrow Y$ is $B$-equivariantly locally trivial, it follows that we obtain a morphism
\begin{align*}
 \Psi \colon Y \rightarrow X(\fan)
\end{align*}
given by 
\begin{align*}
 \Psi(gB)= [\Delta_{\varpi_1}(g),\ldots,\Delta_{\varpi_{\rkg}}(g); \qp{1}(g), \ldots, \qp{\rkg}(g)]
\end{align*}
for $gB\in Y$.

\vspace{10pt}

\subsection{Basic properties of $\Psi$}
Recall that there is a $1$-dimensional complex torus $S\subseteq T$ which acts on the Peterson variety $Y$. Recall also that $T/Z$-action on the toric variety $X(\fan)$ given in \eqref{eq: toric action on Xsigma} is the canonical torus action (See Proposition~\ref{prop: canonical torus action}). To relate these two actions, let us consider the following homomorphism.
\begin{align}\label{eq: equivariance 50'}
 \psi\colon S \rightarrow T/Z
 \quad ; \quad
 t\mapsto [t^{-2}].
\end{align}
Then the following holds.

\begin{proposition}\label{prop: equivariance}
The morphism $\Psi \colon Y \rightarrow X(\fan)$ is equivariant with respect to the homomorphism \eqref{eq: equivariance 50'}.
\end{proposition}

\begin{proof}
For $t\in S\subseteq T$ and $gB\in Y$, the image $\Psi(tgB)\in X(\fan)$ is equal to
\begin{align*}
 [\varpi_1(t)\Delta_{\varpi_1}(g),\ldots,\varpi_{\rkg}(t)\Delta_{\varpi_{\rkg}}(g); \alpha_1(t^{-1})\qp{1}(g), \ldots, \alpha_{\rkg}(t^{-1})\qp{\rkg}(g)] 
\end{align*}
by Lemma~\ref{lem: double equivalence} (2) and Lemma~\ref{lem: single equivalence} (2).
Since the quotient $X(\fan)=(\C^{2I}-\EL)/T$ is taken by the $T$-action given \eqref{eq: def of T-action on C2r}, this is equal to
\begin{align*}
 [\Delta_{\varpi_1}(g),\ldots,\Delta_{\varpi_{\rkg}}(g); \alpha_1(t^{-2})\qp{1}(g), \ldots, \alpha_{\rkg}(t^{-2})\qp{\rkg}(g)] .
\end{align*}
Now, by \eqref{eq: toric action on XSigma}, this coincides with $[t^{-2}]\cdot \Psi(gB)\in X(\fan)$, as desired.
\end{proof}

\vspace{10pt}

The following property will be important when we study the homomorphism on the cohomology rings induced by $\Psi$.

\begin{proposition}\label{prop: surjective Psi}
The morphism $\Psi\colon Y\rightarrow X(\fan)$ is surjective. 
\end{proposition}

\begin{proof}
Since $Y$ is projective, the image $\Psi(Y)$ is a closed subset of $X(\fan)$.
Thus, it suffices to show that the morphism $\Psi\colon Y\rightarrow X(\fan)$ is dominant.
For this purpose, recall that $B^-$ is the opposite Borel subgroup of $G$ with respect to $B$ and that $U^-$ is its unipotent radical.
This allows us to consider an open affine subvariety $Y\cap U^-B/B$ in $Y$.
We also consider an open subvariety of $X(\fan)$ given by
\begin{align*}
 X(\fan)_0\coloneqq \{[x_1,\ldots,x_{\rkg};y_1,\ldots,y_{\rkg}]\in X(\fan) \mid x_i\ne0\ (i\in I)\}.
\end{align*}
By the surjectivity of the isomorphism
\begin{align*}
 T \rightarrow (\C^{\times})^{\rkg}
 \quad ; \quad
 t\mapsto (\varpi_1(t),\ldots,\varpi_{\rkg}(t)), 
\end{align*}
we can write
\begin{align*}
 X(\fan)_0= \{[1,\ldots,1;y_1,\ldots,y_{\rkg}]\in X(\fan) \mid y_i\in \C\ (i\in I)\}.
\end{align*}
For an element $uB\in Y\cap U^-B/B$ with $u\in U^-$, we have
\begin{align*}
 \Delta_{\varpi_i}(uB) = (uv_{\varpi_i})_{\varpi_i} = 1
 \quad \text{for $i\in I$}
\end{align*}
since $(uv_{\varpi_i})_{\varpi_i}$ is the coefficient of the highest weight vector $v_{\varpi_i}$ for the weight decomposition of $uv_{\varpi_i}\in V(\varpi_i)$.
Thus, the morphism $\Psi\colon Y\rightarrow X(\fan)$ restricts to 
\begin{align*}
 \Psi_{0} \colon Y\cap U^-B/B \rightarrow X(\fan)_0
\end{align*}
which sends 
\begin{align}\label{eq: open restriction 2}
 uB^- \ (u\in U^-) \mapsto [1,\ldots,1;\qp{1}(u),\ldots,\qp{\rkg}(u)].
\end{align}
To show that the morphism $\Psi\colon Y\rightarrow X(\fan)$ is dominant, 
it suffices to show that $\Psi_0$ is dominant since $X(\fan)_0$ is a non-empty open subset of the irreducible variety $X(\fan)$.
Let us prove that $\Psi_0$ is dominant in what follows.
For $i\in I$, we denote by $q_i$ the $i$-th function in the last $\rkg$ components of $\Psi_0$:
\begin{align*}
 q'_i \colon Y\cap U^-B/B \rightarrow \C
 \quad ; \quad
 uB\ (u\in U^-) \mapsto \qp{i}(u)=-(\text{Ad}_{u^{-1}}\nil)_{-\alpha_i}.
\end{align*}
In \cite[Theorem~27]{Kostant}, Kostant proved that $q'_1,\ldots,q'_{\rkg}$ are algebraically independent\footnote{For the completeness of the paper, we provide an alternative proof of his result in Appendix (Lemma~B). We note that our $q'_i$ corresponds to his $q_{i^*}$, where $i\mapsto i^*$ is the automorphism of $I$ induced by $w_0$. } in the coordinate ring $\C[Y\cap U^-B/B]$, where we regard $Y\cap U^-B/B$ as an affine variety. 
Let $\C[x_1,\ldots,x_{\rkg}]$ be the polynomial ring with indeterminates $x_1,\ldots,x_{\rkg}$. Kostant's result means that the $\C$-algebra homomorphism
\begin{align*}
 \C[x_1,\ldots,x_{\rkg}] \rightarrow \C[Y\cap U^-B/B]
\end{align*}
sending $x_i$ to $q'_i$ $(i\in I)$ is injective since a non-zero element of the kernel provides a non-trivial algebraic relation between $q'_1,\ldots,q'_{\rkg}$.
This implies that the morphism
\begin{align*}
 Y\cap U^-B/B \rightarrow \C^{\rkg}
 \quad ; \quad
 uB\ (u\in U^-) \mapsto (q'_{1}(u).\ldots,q'_{\rkg}(u))
\end{align*}
is dominant (\cite[Chap.~II, Exercise~2.18(b)]{Hartshorne}). Since $\Psi_0\colon Y\cap U^-B/B \rightarrow X(\fan)_0$ is given by \eqref{eq: open restriction 2}, we conclude that $\Psi_0$ is dominant as well. This completes the proof.
\end{proof}

\vspace{10pt}

\begin{remark}
{\rm
The functions $q'_1,\ldots,q'_{\rkg}$ play the role of quantum parameters in Peterson's description of the quantum cohomology ring of the Langlands dual flag variety $G^{\vee}/B^{\vee}$ (e.g.\ \cite[Sect.~3.3.7]{Rietsch08}).
We note that Kostant's formula for $q'_i$ given in \cite[Theorem~27]{Kostant} is naturally encoded in the morphism $\Psi$ in the following sense.
We take a representative $\rep{w}_0\in N(T)$ of the longest element $w_0\in W$ so that $\f{i}=-\text{Ad}_{\rep{w}_0^{-1}}\e{i^*}$ holds for $i\in I$ (cf.\ \cite[Equation~(13)]{Kostant}), where $i\mapsto i^*$ is the automorphism of $I$ induced by $w_0$.
An arbitrary element $gB\in Y\cap U^-B\cap U\rep{w}_0 B$ can be written as
\begin{align}\label{eq: identity and longest}
 gB = uB = \widetilde{u}\rep{w}_0 B
\end{align}
for some $u\in U^-$ and $\widetilde{u}\in U$. 
The equality $\Psi(uB) =\Psi(\widetilde{u}\rep{w}_0 B)$ implies that
\begin{align*}
 [1,\ldots,1;q'_{1}(u),\ldots,q'_{\rkg}(u)] = [\Delta_{\varpi_1}(\widetilde{u}\rep{w}_0),\ldots,\Delta_{\varpi_{\rkg}}(\widetilde{u}\rep{w}_0);1,\ldots,1] 
 \quad \text{in $X(\fan)$}
\end{align*}
since $\Delta_{\varpi_i}(u)=(uv_{\varpi_i})_{\varpi_i}=1$ and $\qp{i}(\widetilde{u}\rep{w}_0)=-(\text{Ad}_{(\widetilde{u}\rep{w}_0)^{-1}}\nil)_{-\alpha_i}=1$ for $i\in I$, where the latter equality follows from the choice of the representative $\rep{w}_0$ explained above.
Since we have $X(\fan)=(\C^{2I}-\EL)/T$, this means that there exists $t\in T$ such that 
\begin{align}\label{eq: identity and longest 2}
 \alpha_i(t)=q'_{i}(u)
 \quad \text{and} \quad
 \varpi_i(t)=\Delta_{\varpi_i}(\widetilde{u}\rep{w}_0)^{-1} 
\end{align}
for $i\in I$. 
We may write $\Delta'_{i}(u)\coloneqq \Delta_{\varpi_i}(\widetilde{u}\rep{w}_0)$ for $i\in I$ since $\widetilde{u}$ is uniquely determined by $u$ in \eqref{eq: identity and longest}.
Let $C\coloneqq (c_{i,j})_{i,j\in I}$ be the Cartan matrix of $I$. 
Then we have $\alpha_i(t) = \prod_{j\in I} \varpi_j(t)^{c_{i,j}}$
which becomes Kostant's formula (\cite[Theorem~27]{Kostant}) for $q'_{i}$ by \eqref{eq: identity and longest 2}:
\begin{align*}
 q'_{i}(u) = \prod_{j\in I} \Delta'_{j}(u)^{-c_{i,j}}
 \qquad (i\in I)
\end{align*}
for $u\in U^-$ such that $uB\in Y\cap U^-B\cap U\rep{w}_0 B$.
}
\end{remark}

\subsection{Induced ring homomorphism}
The morphism $\Psi \colon Y \rightarrow X(\fan)$ induces a ring homomorphism
\begin{align*}
 \Psi^* \colon H^*(X(\fan);\Q) \rightarrow H^*(Y;\Q)
\end{align*}
on the singular cohomology rings with $\Q$ coefficients considered with the classical topology of $X(\fan)$ and $Y$.
The aim of this subsection is to prove the following.
\begin{theorem}\label{thm: main}
The induced map
\begin{align*}
 \Psi^* \colon H^*(X(\fan);\Q) \rightarrow H^*(Y;\Q)
\end{align*}
is a ring isomorphism.
\end{theorem}

\begin{proof}
We first consider the induced homomorphism
\begin{align*}
 \Psi_* \colon H_*(Y;\Q) \rightarrow H_*(X(\fan);\Q)
\end{align*}
on the homology groups.
Recall that we have $\dim_{\C} Y=\dim_{\C} X(\fan)=\rkg$, where $\rkg$ is the rank of $G$.
Having Proposition~\ref{prop: surjective Psi} in mind, let $d\in\Z_{>0}$ be the degree of the map $\Psi \colon Y\rightarrow X(\fan)$, i.e., 
\begin{align}\label{eq: degree of Psi}
 \Psi_*\cycle{Y} = d\cycle{X(\fan)}
 \qquad \text{in $H_{2\rkg}(X(\fan);\Q)$}
\end{align} 
(\cite[Lemma~19.1.2]{Ful98}),
where $\cycle{Y}\in H_{2\rkg}(Y;\Q)$ and $\cycle{X(\fan)}\in H_{2\rkg}(X(\fan);\Q)$ denote the fundamental cycles of $Y$ and $X(\fan)$, respectively.
Let $0\le k\le 2\rkg$ be an integer.
By the projection formula for the cap product, we have
\begin{align*}
 \Psi_*(\Psi^*(\eta)\cap \cycle{Y}) = \eta\cap \Psi_*\cycle{Y}
 \qquad \text{for $\eta\in H^{k}(X(\fan);\Q)$}.
\end{align*} 
Namely, we have the following commutative diagram
\begin{align}\label{eq: projection formula diagram}
\begin{matrix}
\xymatrix{
H^{k}(X(\fan);\Q) \ar[d]_{\cap\hspace{1pt}\Psi_*\cycle{Y}}\ar[r]^{\Psi^*} & H^{k}(Y;\Q) \ar[d]^{\cap\cycle{Y}} \\
H_{2\rkg-k}(X(\fan);\Q) & H_{2\rkg-k}(Y;\Q) \ar[l]_{\ \ \ \Psi_*},
}
\end{matrix}
\end{align}
where we have $\Psi_*\cycle{Y}=d\cycle{X}$ with $d>0$ by the definition of $d$ above.
Since $X(\fan)$ is a simplicial projective toric variety, the Poincar\'{e} duality holds in $\Q$ coefficients. That is, the cap product map
\begin{align}\label{eq: toric Poincare duality}
 H^{k}(X(\fan);\Q) \rightarrow H_{2\rkg-k}(X(\fan);\Q)
 \quad ; \quad 
 \eta \mapsto \eta\cap\cycle{X(\fan)}
\end{align} 
is an isomorphism (\cite[Sect.~12.4]{CLS}). Hence, $d>0$ implies that the left vertical map in \eqref{eq: projection formula diagram} is an isomorphism.
Therefore, the commutativity of \eqref{eq: projection formula diagram} implies that the induced map
\begin{align*}
 \Psi^*\colon H^*(X(\fan);\Q) \rightarrow H^*(Y;\Q)
\end{align*} 
is injective.

To prove that $\Psi^*$ is an isomorphism, we compare the dimensions of $H^*(X(\fan);\Q)$ and $H^*(Y;\Q)$ as vector spaces over $\Q$.
Since $X(\fan)$ is a simplicial projective toric variety, $\dim_{\Q} H_*(X(\fan);\Q)$ is equal to the number of full-dimensional cones in $\fan$ (\cite[Theorem~12.3.9 and Theorem~12.3.11]{CLS}).
Hence, by \eqref{eq: dim of maximal cone} and Lemma~\ref{lem: complete 1}, we obtain
\begin{align*}
 \dim_{\Q} H_*(X(\fan);\Q) 
 =|\{J\mid J\subseteq I\}|
 = 2^{\rkg}.
\end{align*} 
Also, \cite[Corollary~4.13]{pre13} and \cite[Proposition~5.2]{ha-ty17} imply that $\dim_{\Q} H^*(Y;\Q)$ is equal to the number of $T$-fixed points of $G/B$ which are contained in $Y$. This number is the same as the cardinality of the fixed point set $Y^S$ (see \eqref{eq: S-fixed point 10}) which is equal to $2^{\rkg}$ by \cite[Proposition~5.8]{ha-ty17}. Hence,
we obtain that
\begin{align*}
 \dim_{\Q} H_*(X(\fan);\Q) =2^{\rkg} = \dim_{\Q} H^*(Y;\Q).
\end{align*} 
Combining this with the injectivity of $\Psi^*\colon H^*(X(\fan);\Q) \rightarrow H^*(Y;\Q)$ proved above, we conclude that $\Psi^*$ is an isomorphism.
\end{proof}

\vspace{5pt}

\begin{remark}\label{rem: generalization}
{\rm 
The Peterson variety is an example of regular Hessenberg varieties.
B$\breve{\text{a}}$libanu-Crooks (\cite{ba-cr}) proved that
their cohomology rings $($with $\Q$ coefficients$)$ are isomorphic to certain subrings of the cohomology ring of the corresponding regular semisimple Hessenberg variety (cf. \cite{AHMMS,AHHM}).
For a particular choice of a Hessenberg space, Horiguchi-Masuda-Shareshian-Song proved that their cohomology rings are also isomorphic to that of certain toric orbifolds in classical Lie types (\cite[Equation~(1.2)]{ho-ma-sh-so}).
}
\end{remark}

\vspace{10pt}

The following corollary is equivalent to the fact that the cohomology ring $H^*(Y;\Q)$ is a Poincar\'{e} duality algebra which is known from \cite[Corollary~1.2]{AHMMS} and \cite[Theorem~D]{ba-cr}. We give an alternative proof of this fact here.

\begin{corollary}\label{cor: PD for Y}
For each integer $0\le k\le 2\rkg$, the cap product map
\begin{align*}
 H^{k}(Y;\Q) \rightarrow H_{2\rkg-k}(Y;\Q)
 \quad ; \quad 
 \xi \mapsto \xi\cap\cycle{Y}
\end{align*} 
is an isomorphism, where $\cycle{Y}\in H_{2\rkg}(Y;\Q)$ is the fundamental cycle of $Y$.
\end{corollary}

\begin{proof}
By Theorem~\ref{thm: main}, the induced maps
\begin{align*}
 \Psi^*\colon H^*(X(\fan);\Q) \rightarrow H^*(Y;\Q)
 \quad \text{and}\quad
 \Psi_*\colon H_*(Y;\Q) \rightarrow H_*(X(\fan);\Q)
\end{align*} 
are isomorphisms.
Thus, the claim follows from \eqref{eq: degree of Psi} and the commutativity of the diagram \eqref{eq: projection formula diagram} since $d>0$.
\end{proof}

\vspace{10pt}

\subsection{Correspondences on degree $2$}\label{subsec: corresponding on degree 2}
In the previous subsection, we proved that the induced map $\Psi^*\colon H^*(X(\fan);\Q) \rightarrow H^*(Y;\Q)$ is a ring isomorphism.
In this subsection, we describe the isomorphism on degree $2$
\begin{align}\label{eq: induced on degree 2}
 \Psi^*\colon H^2(X(\fan);\Q) \rightarrow H^2(Y;\Q)
\end{align} 
explicitly.

The Peterson variety $Y$ admits line bundles $\LB_{\lambda}|_Y$ for $\lambda\in\WL$ (see Section~\ref{subsec: line bundles}), where $\WL$ is the weight lattice of $T$.
For simplicity, we write this line bundle as $\LB_{\lambda}$
by omitting the symbol $|_Y$ in the rest of this paper.
Thus, we have Chern classes $c_1(\LB_{\lambda})\in H^2(Y;\Q)$. 
On the toric variety $X(\fan)$, we have the torus invariant prime divisors $\Diva{i}$ and $\Divw{i}$ corresponding to the rays in $\fan$ generated by $-\alpha^{\vee}_i$ and $\cvarpi_i$, respectively. 
Regarding them as subvarieties of codimension $1$, they determine homology cycles $\cycle{\Diva{i}},\cycle{\Divw{i}}\in H_{2\rkg-2}(X(\fan);\Q)$ (\cite[Sect.19.1]{Ful98}).
Since the Poincar\'{e} duality map \eqref{eq: toric Poincare duality} is an isomorphism, we obtain cohomology classes $\PD{\Diva{i}},\PD{\Divw{i}}\in H^2(X(\fan);\Q)$ which satisfy
\begin{align}\label{eq: poincare duality}
 \PD{\Diva{i}}\cap\cycle{X(\fan)} = \cycle{\Diva{i}}
 , \quad
 \PD{\Divw{i}}\cap\cycle{X(\fan)} = \cycle{\Divw{i}}.
\end{align} 
In what follows, we describe the relation between $\PD{\Diva{i}},\PD{\Divw{i}}\in H^2(X(\fan);\Q)$ and 
$c_1(\LB_{\lambda})\in H^2(Y;\Q)$ under the isomorphism \eqref{eq: induced on degree 2}.

\begin{remark}
{\rm 
In \cite[Sect.~12.4]{CLS}, the cohomology class associated to a torus invariant prime divisor $D$ on a simplicial toric variety is given by $\ell^{-1}c_1(L_{\ell D})$, where $\ell\in\Z_{>0}$ is a positive integer such that $\ell D$ is Cartier and  $L_{\ell D}$ is the corresponding line bundle. See \cite[Proposition~12.4.13~(d)]{CLS} and its proof for details.
This agrees with our construction in \eqref{eq: poincare duality} by \cite[Example~2.5.5 and Proposition~19.1.2]{Ful98}.}
\end{remark}

\vspace{10pt}

We begin with the following claim. 

\begin{lemma}\label{lem: proportional lemma}
For $i\in I$, there exists $m_i\in\Q$ such that
\begin{align*}
 m_i \Psi^*(\PD{\Divw{i}}) = c_1(\LB_{\alpha_i}) \in H^2(Y;\Q).
\end{align*}
\end{lemma}

\begin{proof}
Recalling the construction of the morphism $\Psi\colon Y\rightarrow X(\fan)$ and \eqref{eq: desired section 1}, it sends the zero locus $Z(\phi_{\alpha_i})\subseteq Y$ to $\Divw{i}\subseteq X(\fan)$;
\begin{align}\label{eq: pushing forward zero schemes 2}
 \Psi (Z(\phi_{\alpha_i})) \subseteq \Divw{i} 
\end{align}
by \eqref{eq: invariant divisor}.
Let $\mathcal{Z}(\phi_{\alpha_i})$ be the zero scheme defined by the section $\phi_{\alpha_i} \colon Y\rightarrow \LB_{\alpha_i}|_Y$ whose underlying closed subset is $Z(\phi_{\alpha_i})$.
Since we have $\dim_{\C}Y=\rkg$, this scheme determines a cycle $\cycle{\mathcal{Z}(\phi_{\alpha_i})}$ in the homology group $H_{2(\rkg-1)}(Y;\Q)$, and \eqref{eq: pushing forward zero schemes 2} implies that 
\begin{align}\label{eq: pushing forward zero schemes}
 \Psi_*(\cycle{\mathcal{Z}(\phi_{\alpha_i})}) = \ell_i \cycle{\Divw{i}}
 \qquad \text{in $H_{2(\rkg-1)}(X(\fan);\Q)$}
\end{align}
for some $\ell_i\in\Z_{\ge0}$ (see Lemma C in Appendix).
Also, since $Y$ is Cohen-Macaulay (\cite[Corollary~3.8]{ab-fu-ze},\cite{Peterson}), we have 
\begin{align}\label{eq: PD of zero scheme 1}
 \cycle{\mathcal{Z}(\phi_{\alpha_i})} = c_1(\LB_{\alpha_i})\cap\cycle{Y}
\end{align}
by \cite[Proposition~14.1 and Example~14.1.1]{Ful98}.
Recall from \eqref{eq: degree of Psi} that the degree of the map $\Psi\colon Y\rightarrow X(\fan)$ was denoted by $d>0$.
Under the induced map $\Psi_*$ on the homology groups, the cap product $\ell_i \Psi^*(\PD{\Divw{i}})\cap\cycle{Y}\in H_{2(\rkg-1)}(Y;\Q)$ is sent to
\begin{align*}
 \Psi_*\Big(\ell_i \Psi^*(\PD{\Divw{i}})\cap\cycle{Y}\Big) 
 &= \ell_i \PD{\Divw{i}}\cap\Psi_*\cycle{Y} \qquad \text{(by the projection formula)}\\
 &= \ell_i (\PD{\Divw{i}})\cap d\cycle{X(\fan)} \\
 &= d \ell_i \cycle{\Divw{i}} \hspace{80pt} \text{(by \eqref{eq: poincare duality})}\\
 &= \Psi_*\Big( d c_1(\LB_{\alpha_i})\cap \cycle{Y}\Big) \qquad \text{(by \eqref{eq: pushing forward zero schemes} and \eqref{eq: PD of zero scheme 1})}.
\end{align*} 
By Theorem~\ref{thm: main}, 
the induced map $\Psi_*$ on homology groups is also an isomorphism so that we obtain from this computation that
\begin{align*}
 \ell_i \Psi^*(\PD{\Divw{i}})\cap\cycle{Y}
 = d c_1(\LB_{\alpha_i})\cap \cycle{Y}.
\end{align*} 
By Corollary~\ref{cor: PD for Y}, this further implies that
\begin{align*}
 \ell_i \Psi^*(\PD{\Divw{i}})
 = d c_1(\LB_{\alpha_i}).
\end{align*} 
Since $d>0$, we  obtain the  desired claim.
\end{proof}

\vspace{10pt}

We next prove that $m_i=1$ for all $i\in I$.
For that purpose, we use the torus equivariant cohomology rings. One may finds detail explanations on general properties of equivariant cohomology rings in \cite{AF-Equiv} and \cite[Sect.~12.4]{CLS}.

Let $\trs$ be a complex torus so that $\trs\cong (\C^{\times})^r$ for some $r\in\Z_{>0}$.
Let $\pi\colon E\trs\rightarrow B\trs$ be a universal principal $\trs$-bundle whose total space $E\trs$ is contractible.
For a topological $\trs$-space $X$, the $\trs$-equivariant cohomology of $X$ is defined by
\begin{align*}
 H_{\trs}^*(X;\Q) \coloneqq H^*(E\trs\times_{\trs} X;\Q), 
\end{align*} 
where $E\trs\times_{\trs} X$ is the quotient space of the product $E\trs\times X$ for the $\trs$-action given by $k\cdot (u,x)\coloneqq (uk^{-1}, kx)$ for $k\in \trs$ and $(u,x)\in E\trs\times X$.
We denote by $\text{pt}$ a one point space with the trivial $\trs$-action.
Then we have $E\trs\times_{\trs} \text{pt}=B\trs\times \text{pt}$ so that 
$H_{\trs}^*(\text{pt};\Q) = H^*(B\trs;\Q)$.
With this identification, the map $X\rightarrow \text{pt}$ induces a ring homomorphism
\begin{align*}
 H^*(B\trs;\Q)\rightarrow H_{\trs}^*(X;\Q)
\end{align*} 
which allows us to regard $H_{\trs}^*(X;\Q)$ as an algebra over $H^*(B\trs;\Q)$.
So it would be helpful for us to clarify the structure of $H^*(B\trs;\Q)$.
For that purpose, let us set $M\coloneqq \Hom(\trs,\C^{\times})$.
For each $\lambda\in M$, there is a one-dimensional representation $\C_{\lambda}(=\C)$ of $\trs$ given by
\begin{align*}
 k\cdot z \coloneqq \lambda(k)z
\end{align*} 
for $k\in \trs$ and $z\in\C_{\lambda}$. This gives rise to an associated complex line bundle 
\begin{align*}
 E\trs\times_{\trs}\C_{\lambda} \rightarrow B\trs
 \quad ; \quad
 [u,z]\rightarrow \pi(u), 
\end{align*} 
where $[u,z]\in E\trs\times_{\trs}\C_{\lambda}$ is the class represented by $(u,z)\in E\trs\times\C_{\lambda}$.
By taking the first Chern class of this line bundle, we obtain a map
\begin{align}\label{eq: sym alg and coh BK 1}
 M(=\Hom(\trs,\C^{\times})) \rightarrow H^2(B\trs;\Q)
 \quad ; \quad
 \lambda \rightarrow c_1(E\trs\times_{\trs}\C_{\lambda}).
\end{align} 
We denote by $\text{Sym}_{\Q}(M)$ the symmetric algebra on the the vector space $M\otimes_{\Z}\Q$ over $\Q$. Then
it is well-known that this map induces a ring isomorphism
\begin{align}\label{eq: sym alg and coh BK 2}
 \text{Sym}_{\Q}(M) \rightarrow H^*(B\trs;\Q)
\end{align} 
which sends $\lambda\in M$ to $c_1(E\trs\times_{\trs}\C_{\lambda})\in H^2(B\trs;\Q)$.
In the rest of this paper, we identify $\text{Sym}_{\Q}(M)$ and $H^*(B\trs;\Q)$ under this isomorphism\footnote{Some authors identify $\text{Sym}_{\Q}(M)$ and $H^*(B\trs;\Q)$ by putting a sign to \eqref{eq: sym alg and coh BK 1} (e.g.\ \cite[Sect.~12.4]{CLS}).}.
\\

We now return to our situation to prove that $m_i=1$ for all $i\in I$.
Recall that the morphism $\Psi\colon Y\rightarrow X(\fan)$ is equivariant with respect to the homomorphism $\psi\colon S\rightarrow T/Z$ of tori given in \eqref{eq: equivariance 50'}.
Thus, it induces an algebra-homomorphism
\begin{align*}
 \Psi^* \colon H_{T/Z}^*(X(\fan);\Q) \rightarrow H_S^*(Y;\Q)
\end{align*} 
with respect to the ring homomorphism
\begin{align*}
 \psi^* \colon H^*(B(T/Z);\Q)\rightarrow H^*(BS;\Q).
\end{align*}
As we saw in \eqref{eq: def of alpha map}, we have a character of $S(\subseteq T)$ given by 
\begin{align*}
 \alpha_S \colon S\rightarrow \C^{\times}
 \quad ; \quad 
 t\mapsto \alpha_S(t)=\alpha_1(t)=\cdots=\alpha_{\rkg}(t).
\end{align*}
Thus, by \eqref{eq: equivariance 50'}, we can express the above map $\psi^*\colon H^2(B(T/Z);\Q)\rightarrow H^2(BS;\Q)$ on degree $2$ as
\begin{align}\label{eq: S and C* 20}
 \psi^*(\alpha_i)= -2\alpha_S(=-2\alpha_1=\cdots=-2\alpha_{\rkg})
 \qquad (i\in I).
\end{align}
Since $S$ is a $1$ dimensional torus and the map $\alpha_S \colon  S\rightarrow \C^{\times}$ has at most a finite kernel, it follows that $H^*(BS;\Q)$ is the polynomial ring freely generated by $\alpha_S$, i.e., $H^*(BS;\Q)=\Q[\alpha_S]$.

\begin{remark}\label{rem: ha-ho-ma generator}
{\rm 
The generator $\alpha_S\in H^2(BS;\Q)$ is denoted by $t$ in  \cite{ha-ho-ma}.}
\end{remark}

\vspace{10pt}

We have the following commutative diagram induced by the morphism $\Psi\colon Y\rightarrow X(\fan)$ whose vertical arrows are the forgetful maps:
\begin{equation}\label{eq: forgetful diagram}
\begin{split}
\xymatrix{
H_{T/Z}^2(X(\fan);\Q) \ar[d]_{}\ar[r]^{\ \ \ \Psi^*} & H_S^2(Y;\Q) \ar[d]^{} \\
H^2(X(\fan);\Q) \ar[r]^{\ \ \Psi^*} & H^2(Y;\Q).
}
\end{split}
\end{equation}
The torus invariant prime divisor $\Divw{i}$ on $X(\fan)$ determines an equivariant cohomology class $$\PD{\Divw{i}}_{T/Z}\in H_{T/Z}^2(X(\fan);\Q)$$ (see \cite[Proof of Proposition~12.4.13]{CLS}) which is sent to the ordinary cohomology class $\PD{\Divw{i}}\in H^2(X(\fan);\Q)$ under the left vertical map in \eqref{eq: forgetful diagram}.
Also, since the line bundle $\LB_{\alpha_i}$ over $G/B$ is $G$-equivariant (see \eqref{eq:def of L_lambda 3}), its restriction over $Y$ is an $S$-equivariant line bundle.
Thus, it has the $S$-equivariant first Chern class $$c_1^S(\LB_{\alpha_i})\in H_S^2(Y;\Q)$$ which is defined to be the first Chern class of the induced line bundle $ES\times_S \LB_{\alpha_i}\rightarrow ES\times_S Y$.
It is sent to the (ordinary) first Chern class $c_1(\LB_{\alpha_i})\in H^2(Y;\Q)$ under the right vertical map in \eqref{eq: forgetful diagram}.

Recall that Lemma~\ref{lem: proportional lemma} states an equality in $H^2(Y;\Q)$ sitting in the bottom right of \eqref{eq: forgetful diagram}.
We lift and study this equality in the $S$-equivariant cohomology $H_S^2(Y;\Q)$. For that purpose, let us describe some basic properties of $H_S^*(Y;\Q)$.
By \cite{pre13}, $Y$ admits an affine paving so that we have $H^{\text{odd}}(Y;\Q)=0$.
Thus, the following holds (e.g.\ \cite[III.\ Theorem~4.2]{mi-to}):
the structure map (as an $H^*(BS;\Q)$-algebra)
\begin{align}\label{eq: algebra structure map}
 H^*(BS;\Q)\rightarrow H_S^*(Y;\Q)
\end{align}
is injective and the forgetful map 
\begin{align*}
 H_S^*(Y;\Q)\rightarrow H^*(Y;\Q)
\end{align*}
is a surjective homomorphism whose kernel is generated by the image of  $H^{>0}(BS;\Q)$ under the structure map \eqref{eq: algebra structure map}.
We also denote by $\alpha_S\in H_S^2(Y;\Q)$ the image of $\alpha_S\in H^2(BS;\Q)$ under the injective map \eqref{eq: algebra structure map}.

By the commutativity  of the diagram \eqref{eq: forgetful diagram}, it follows that its right vertical map sends $\Psi^*(\PD{\Divw{i}}_{T/Z})$ to $\Psi^*(\PD{\Divw{i}}))$. Combining this with Lemma~\ref{lem: proportional lemma}, we see that the difference $m_i\Psi^*(\PD{\Divw{i}}_{T/Z})-c_1^S(\LB_{\alpha_i})\in H_S^2(Y;\Q)$ belongs to its kernel.
Therefore, there exists $n_i\in\Q$ such that 
\begin{align}\label{eq: equiavriant linear formula}
m_i\Psi^*(\PD{\Divw{i}}_{T/Z})=c_1^S(\LB_{\alpha_i})+n_i\alpha_S \quad \text{in $H_S^2(Y;\Q)$}.
\end{align}

Recall that $\rep{w}_J B\in Y$ is the $S$-fixed point corresponding to $J\subseteq I$ (see \eqref{eq: S-fixed point of Y}). 
Let $p_J\in X(\fan)$ be the torus fixed point corresponding to the maximal cone $\sigma_J$ (see \eqref{eq: def of maximal cone}) for $J\subseteq I$. Then the following holds.

\begin{lemma}\label{lem: image of fixed points}
We have $\Psi(\rep{w}_J B)=p_J$ for $J\subseteq I$.
\end{lemma}

\begin{proof}
Recall that the components of the morphism $\Psi$ are given by the functions which produces the sections $\psi_{\varpi_i}$ and $\phi_{\alpha_i}$ of the line bundles $\LB_{\varpi_i}|_Y$ and $\LB_{\alpha_i}|_Y$, respectively.
For these sections, we have
\begin{align*}
 \psi_{\varpi_i}(\rep{w}_J B)=[\rep{w}_J,\Delta_{\varpi_i}(\rep{w}_J)]=0 \quad \text{for $i\in J$}
\end{align*} 
by \eqref{eq: psi i by ()} and \eqref{eq: fixed point 1}, and we have
\begin{align*}
 \phi_{\alpha_i}(\rep{w}_J B)=[\rep{w}_J ,\qp{i}(\rep{w}_J)]=0 \quad \text{for $i\notin J$}
\end{align*} 
by \eqref{eq: desired section 1} and \eqref{eq: fixed point 2}.
This implies that we can write 
\begin{align*}
 \Psi(\rep{w}_J B) = [x_1,\ldots,x_{\rkg};y_1,\ldots,y_{\rkg}]
\end{align*} 
with $x_i=0$ for $i\in J$ and $y_i=0$ for $i\notin J$.
Recalling the description of the torus invariant divisors $\Diva{i}$ and $\Divw{i}$ on $X(\fan)$ from \eqref{eq: invariant divisor}, it follows that $\Psi(\rep{w}_J B)$ is the intersection of the divisors 
$\Diva{i}$ for $i\in J$ and $\Divw{i}$ for $i\notin J$, 
and hence it must be the torus fixed point $p_J\in X(\fan)$ corresponding to the maximal cone $\sigma_J=\ctc{J}{I-J}$ (see \eqref{eq: def of maximal cone}). Namely, we obtain that $\Psi(\rep{w}_J B) = p_J\in X(\fan)$.
\end{proof}

\vspace{10pt}

We denote by $i_J\colon \{p_J\}\hookrightarrow X(\fan)$ and $j_J\colon \{\rep{w}_J B\}\hookrightarrow Y$ the inclusion maps. 
Since $\Psi(\rep{w}_J B)=p_J$ by Lemma~\ref{lem: image of fixed points}, we have the following commutative diagram
\begin{equation*}
\begin{split}
\xymatrix{
H_{T/Z}^2(X(\fan);\Q) \ar[d]_{i_J^*}\ar[r]^{\quad\Psi^*} & H_S^2(Y;\Q) \ar[d]^{j_J^*} \\
H^2(B(T/Z);\Q)
\ar[r]^{\quad\psi^*} & 
H^2(BS;\Q),
}
\end{split}
\end{equation*}
where we identified the equivariant cohomology rings of fixed points ($p_J$ and $\rep{w}_J B$) and the cohomology rings of classifying spaces ($B(T/Z)$ and $BS$, respectively).
Namely, we have 
\begin{align*}
 \psi^*\circ i_J^*(\PD{\Divw{i}}_{T/Z}) = j_J^*\circ\Psi^*(\PD{\Divw{i}}_{T/Z})
 \quad \text{in $H^2(BS;\Q)$}.
\end{align*}
After multiplying $m_i$ to the both sides of this equality, 
we can compute the right hand side by \eqref{eq: equiavriant linear formula} so that we obtain
\begin{align}\label{eq: partial equality}
 m_i \psi^*\circ i_J^*(\PD{\Divw{i}}_{T/Z}) = j_J^*(c_1^S(\LB_{\alpha_i}))+n_i \alpha_S
 \quad \text{in $H^2(BS;\Q)$},
\end{align}
where we used the fact that the composition of the structure map \eqref{eq: algebra structure map} and $j_J^*$ is the identity map on $H^2(BS;\Q)$.
The next lemma computes $i_J^*(\PD{\Divw{i}}_{T/Z})$ in \eqref{eq: partial equality} for the cases $J=\emptyset$ and $J=I$.

\begin{lemma}\label{lem: localization to eB and w0B toric}
For $i\in I$, we have
\begin{align*}
i_{\emptyset}^*(\PD{\Divw{i}}_{T/Z}) = \alpha_i 
\quad  \text{and} \quad
i_I^*(\PD{\Divw{i}}_{T/Z})= 0.
\end{align*}
\end{lemma}

\begin{proof}
For the first equality, the fixed point $p_{\emptyset}\in X(\fan)$ is the torus orbit corresponding to the maximal cone $\sigma_{\emptyset}=\ctc{\emptyset}{I}=\text{cone}( \{\cvarpi_k \mid k\in I \} )$. 
The dual basis of the primitive generators $\{\cvarpi_k \mid k\in I \}\subset \CWL$ of this cone is the set of simple roots $\{\alpha_k \mid k\in I\}\subset \WL$. Thus, by the proof of \cite[Lemma~12.4.17]{CLS}, 
the image $i_{\emptyset}^*(\PD{\Divw{i}}_{T/Z})$ is the $i$-th dual basis $\alpha_i$, where we note that \cite{CLS} uses an identification $\mathsf{s}\colon \text{Sym}_{\Q}(M)\rightarrow H^*(B(T/Z);\Q)$ with $M=\Hom(T/Z,\C^{\times})$ induced by the linear map \eqref{eq: sym alg and coh BK 1} (for the torus $T/Z$) times $-1$.

For the second equality, the fixed point $p_{I}\in X(\fan)$ is the torus orbit corresponding to the maximal cone $\sigma_{I}=\ctc{I}{\emptyset}=\text{cone}( \{-\alpha^{\vee}_j \mid \ j\in I \} )$.
The divisor $\Divw{i}$ in the claim is the closure of the torus orbit corresponding to $\ctc{\emptyset}{\{i\}}=\text{cone}( \cvarpi_i )$. For these cones, we have $\ctc{\emptyset}{\{i\}}\cap \sigma_{I}=\ctc{\emptyset}{\emptyset}=\{\bm{0}\}$ by Lemma~\ref{lem: intersection of two cones}, and hence $\ctc{\emptyset}{\{i\}}\not\subseteq \sigma_{I}$ which means $p_I\notin \Divw{i}$ by \cite[Theorem~3.2.6~(d)]{CLS}.
Hence, we obtain the second equality 
(\cite[Proposition~12.4.13 (c) and equation (12.4.16)]{CLS}).
\end{proof}

\vspace{10pt}

The next lemma computes $j_J^*(c_1^S(\LB_{\alpha_i}))$ in \eqref{eq: partial equality} for the cases $J=\emptyset$ and $J=I$.

\begin{lemma}\label{lem: localization to eB and w0B}
For $i\in I$, we have
\begin{align*}
j_{\emptyset}^*(c_1^S(\LB_{\alpha_i})) = -\alpha_S
\quad  \text{and} \quad
j_I^*(c_1^S(\LB_{\alpha_i}))= \alpha_S.
\end{align*}
\end{lemma}

\begin{proof}
We first prove the second equality.
By definition, the $S$-equivariant first Chern class $c_1^S(\LB_{\alpha_i})$ is the first Chern class of the line bundle $$ES\times_S \LB_{\alpha_i}\rightarrow ES\times_S Y.$$ 
Hence, the pullback $j_I^*(c_1^S(\LB_{\alpha_i}))$ under the inclusion $j_I\colon \{\rep{w}_I B\}\hookrightarrow Y$ is the first Chern class of the line bundle 
\begin{align}\label{eq: line bundle over BS}
ES\times_S (\LB_{\alpha_i}|_{\rep{w}_I B})
\rightarrow 
ES\times_S \{\rep{w}_I B\}\cong BS,
\end{align}
where $\LB_{\alpha_i}|_{\rep{w}_I B}(\cong\C)$ is the fiber of $\LB_{\alpha_i}$ over the point $\rep{w}_I B\in Y$. 
Recalling the identification of $M\otimes_{\Z}\Q$ and $H^2(BS;\Q)$ given in \eqref{eq: sym alg and coh BK 2}, the first Chern class of the line bundle \eqref{eq: line bundle over BS} is the weight of the 1-dimensional $S$-representation $\LB_{\alpha_i}|_{\rep{w}_I B}$. So let us compute this weight.
For $t\in S$ and $[\rep{w}_I,z]\in \LB_{\alpha_i}|_{\rep{w}_I B}$, we have
\begin{align*}
 t\cdot [\rep{w}_I ,z] = [t\rep{w}_I ,z] = [\rep{w}_I \cdot \rep{w}_I^{-1}t\rep{w}_I ,z]
 =[\rep{w}_I , \alpha_i((\rep{w}_I^{-1}t\rep{w}_I)^{-1})z]
 =[\rep{w}_I , (-w_I\alpha_i)(t),z],
\end{align*}
where the first equality follows from \eqref{eq:def of L_lambda 3} and the third equality follows from \eqref{eq:def of L_lambda 2}.
Since $w_I\in W$ is the longest element, we have $(-w_I\alpha_i)(t)=\alpha_{i^*}(t)=\alpha_S(t)$ since $t\in S$, where $i\mapsto i^*$ is the automorphism of $I$ induced by the longest element $w_I$.
Therefore, we conclude that
\begin{align*}
j_I^*(c_1^S(\LB_{\alpha_i})) = c_1(ES\times_S (\LB_{\alpha_i}|_{\rep{w}_I B})) = \alpha_S .
\end{align*}

The same argument also implies the first equality $j_{\emptyset}^*(c_1^S(\LB_{\alpha_i})) = -\alpha_S$ by replacing the longest element $w_I\in W$ to the identity element $w_{\emptyset}\in W$.
\end{proof}

\begin{remark}
{\rm
Lemma~\ref{lem: localization to eB and w0B} can also be deduced from \cite[Lemma 5.2 and equation (5.7)]{AHMMS}, where we note that our line bundle $\LB_{\lambda}$ is written as $\LB_{-\lambda}$ in \cite{AHMMS} for $\lambda\in\WL=\Hom(T,\C^{\times})$.}
\end{remark}

\vspace{10pt}

By these two lemmas and \eqref{eq: S and C* 20}, the equality \eqref{eq: partial equality} for the case $J=\emptyset$ takes of the form
\begin{align*}
 -2 m_i \alpha_S = (-1+n_i)\alpha_S.
\end{align*}
Similarly, the equality \eqref{eq: partial equality} for the case $J=I$ takes of the form
\begin{align*}
 0 = (1+n_i)\alpha_S.
\end{align*}
Thus, we obtain that $n_i =-1$ and $m_i=1$ for $1\le i\le r$.
Therefore, Lemma~\ref{lem: proportional lemma} now implies the following.

\begin{corollary}\label{cor: = 1 lemma}
For $i\in I$, we have $m_i=1$, that is,
\begin{align*}
 \Psi^*(\PD{\Divw{i}}) = c_1(\LB_{\alpha_i}) \in H^2(Y;\Q).
\end{align*}
\end{corollary}

\vspace{5pt}
We finish this subsection by providing the following claim which compensates the statement of Corollary~\ref{cor: = 1 lemma}.

\begin{corollary}\label{cor: = 1 lemma 2}
For $i\in I$, we have
\begin{align*}
 \Psi^*(\PD{\Diva{i}}) = c_1(\LB_{\varpi_i}) \in H^2(Y;\Q).
\end{align*}
\end{corollary}

\begin{proof}
Since $\Diva{i}$ and $\Divw{i}$ are the torus invariant prime divisors on $X(\fan)$ corresponding to the ray vectors $-\alpha^{\vee}_i$ and $\cvarpi_i$, we have a linear relation
\begin{align*}
 \sum_{j\in I} \langle \alpha_i, -\alpha^{\vee}_j \rangle \PD{\Diva{j}}
 + \sum_{j\in I} \langle \alpha_i, \cvarpi_j \rangle \PD{\Divw{j}} =0
 \qquad \text{in $H^2(X(\fan);\Q)$}
\end{align*} 
for each $i\in I$ (see \cite[Sect.~12]{CLS}).
For the coefficients in this equality, we have $\langle \alpha_i, \cvarpi_j\rangle=\delta_{ij}$ and $\langle \alpha_i, -\alpha^{\vee}_j\rangle=-c_{i,j}$, where $(c_{i,j})_{i,j\in I}$ is the Cartan matrix of $I$.
Thus, this relation can be expressed as
\begin{align}\label{eq: PD relation before 2}
 \PD{\Divw{i}} 
 = \sum_{j\in I} c_{i,j} \PD{\Diva{j}}
 \qquad \text{in $H^2(X(\fan);\Q)$}.
\end{align} 
Recall that we have $\alpha_i = \sum_{j\in I} c_{i,j} \varpi_j$ in the weight lattice $\WL$ which implies that
\begin{align}\label{eq: alpha in pi 2}
 c_1(\LB_{\alpha_i}) = \sum_{j\in I} c_{i,j} c_1(\LB_{\varpi_j}) 
 \qquad \text{in $H^2(Y;\Q)$}.
\end{align}
Since the Cartan matrix $(c_{ij})_{i,j\in I}$ is invertible over $\Q$, Corollary~\ref{cor: = 1 lemma} now implies the desired claim by \eqref{eq: PD relation before 2} and \eqref{eq: alpha in pi 2}.
\end{proof}

\vspace{5pt}

\subsection{Compatibility with explicit presentations for $H^*(X(\fan);\Q)$ and $H^*(Y;\Q)$}\label{subsec: explicit presentations}
Since $X(\fan)$ is a simplicial projective toric variety, there is a well-known presentation of the cohomology ring $H^*(X(\fan);\Q)$ given as a quotient of Stanley-Reisner ring by linear relations (\cite[Sect.\ 12.4]{CLS}).
Also, Harada-Horiguchi-Masuda (\cite{ha-ho-ma}) gave an explicit presentation of the cohomology ring $H^*(Y;\Q)$. 
In this subsection, we explain the compatibility of these presentations and the isomorphism $\Psi^*\colon H^*(X(\fan);\Q) \rightarrow H^*(Y;\Q)$.

We begin with reviewing the presentation of the cohomology ring $H^*(X(\fan);\Q)$ according to \cite[Sect.\ 12.4]{CLS}.
Consider a ring homomorphism
\begin{align}\label{eq: coh generators}
 \Q[X_i,Y_i\mid i\in I] \rightarrow H^*(X(\fan);\Q)
\end{align} 
defined by $X_i\mapsto \PD{\Diva{i}}$ and $Y_i\mapsto \PD{\Divw{i}}$ for $i\in I$.
Let $\mathcal{I},\mathcal{J}\subseteq \Q[X_i,Y_i\mid i\in I]$ be the ideals ogiven by
\begin{align*}
 &\mathcal{I} \coloneqq \langle X_iY_i \mid i\in I \rangle, \\
 &\mathcal{J} \coloneqq \left\langle \left. \sum_{j\in I} \langle \alpha_i , -\alpha^{\vee}_j \rangle X_j + \sum_{j\in I} \langle \alpha_i , \cvarpi_j \rangle Y_j \ \right| \  i\in I \right\rangle 
 = \left\langle \left.  - \sum_{j\in I} c_{i,j} X_j \ + \ Y_i \ \right| \  i\in I \right\rangle,
\end{align*} 
where $(c_{i,j})_{i,j\in I}=(\langle \alpha_i , -\alpha^{\vee}_j \rangle)_{i,j\in I}$ is the Cartan matrix of $I$.
Since the fan $\fan$ is simplicial and projective, 
the map \eqref{eq: coh generators} induces a ring isomorphism (\cite[Theorem~12.4.1]{CLS}):
\begin{align*}
 \Q[X_i,Y_i\mid i\in I]/(\mathcal{I}+\mathcal{J}) \stackrel{\cong}{\rightarrow} H^*(X(\fan);\Q).
\end{align*} 

We may simplify the quotient ring $\Q[X_i,Y_i\mid i\in I]/(\mathcal{I}+\mathcal{J})$ as follows. 
Observe from the definition of the ideal $J$ that we have 
\begin{align*}
 \overline{Y}_i = \sum_{j\in I} c_{i,j} \overline{X}_j  
 \qquad \text{in $\Q[X_i,Y_i\mid i\in I]/(\mathcal{I}+\mathcal{J})$}
\end{align*}
(cf.\ \eqref{eq: PD relation before 2}),
where $\overline{X}_i$ and $\overline{Y}_i$ are the image of $X_i$ and $Y_i$ in the quotient, respectively.
This means that we can rewrite the above isomorphism as
\begin{align}\label{eq: coh generators 2}
 \Q[X_i \mid i\in I]/\mathcal{I}' \stackrel{\cong}{\rightarrow} H^*(X(\fan);\Q)
\end{align} 
which sends $X_i$ to $\PD{\Diva{i}}$,
where $\mathcal{I}'$ is an ideal of $\Q[X_i \mid i\in I]$ given by
\begin{align*}
 \mathcal{I}' \coloneqq \left\langle \left. X_i \sum_{j\in I} c_{i,j} X_j \ \right| \  i\in I \right\rangle.
\end{align*} 
Let $\WL_{\Q} \coloneqq \WL\otimes_{\Z}\Q$, where $\WL$ is the weight lattice of $T$.
Since we have $\alpha_i = \sum_{j\in I}c_{i,j}\varpi_j$ in $\WL$, there is a ring isomorphism $\Q[X_i \mid i\in I]/\mathcal{I}'\cong \Sym \WL_{\Q}\hspace{2pt}/\langle \varpi_i \alpha_i \mid i\in I \rangle$ given by $X_i\mapsto \varpi_i$ $(i\in I)$. 
Therefore, the presentation \eqref{eq: coh generators 2} can be expressed as
\begin{align}\label{eq: coh generators 4}
 \Sym \WL_{\Q} \hspace{2pt} /\langle \varpi_i\alpha_i \mid i\in I\rangle \stackrel{\cong}{\rightarrow} H^*(X(\fan);\Q)
\end{align} 
which sends $\varpi_i$ to $\PD{\Diva{i}}$ for $i\in I$.
This gives us an explicit presentation of the cohomology ring $H^*(X(\fan);\Q)$ in terms of representation theory.

In \cite{ha-ho-ma}, Harada-Horiguchi-Masuda gave an explicit presentation of the cohomology ring $H^*(Y;\Q)$ (see also \cite[Remark in Sect.~4]{ha-ho-ma}) by using the quotient ring appeared in \eqref{eq: coh generators 2}. By identifying this quotient ring with $\Sym \WL_{\Q}\hspace{2pt}/\langle \varpi_i \alpha_i \mid i\in I \rangle$ as above, their presentation can be expressed as
\begin{align}\label{eq: coh generators 3}
 \Sym \WL_{\Q}\hspace{2pt} /\langle \varpi_i\alpha_i \mid i\in I\rangle \stackrel{\cong}{\rightarrow} H^*(Y;\Q)
\end{align} 
which sends $\varpi_i$ to $c_1(\LB_{\varpi_i})\in H^2(Y;\Q)$.

Now, Corollary~\ref{cor: = 1 lemma 2} means that the presentations \eqref{eq: coh generators 4} and \eqref{eq: coh generators 3} are compatible with the isomorphism $\Psi^*\colon H^*(X(\fan);\Q)\rightarrow H^*(Y;\Q)$. Namely, we obtain the following commutative diagram of ring isomorphisms.
\[
{\unitlength 0.1in%
\begin{picture}(18.8000,7.7000)(14.0000,-20.9000)%
\put(14.0000,-16.0000){\makebox(0,0)[lb]{$H^*(X(\fan);\Q)$}}%
\put(30.1000,-16.0000){\makebox(0,0)[lb]{$H^*(Y;\Q)$}}%
%
\special{pn 8}%
\special{pa 2400 1510}%
\special{pa 2900 1510}%
\special{fp}%
\special{sh 1}%
\special{pa 2900 1510}%
\special{pa 2833 1490}%
\special{pa 2847 1510}%
\special{pa 2833 1530}%
\special{pa 2900 1510}%
\special{fp}%
\put(19.0000,-22.2000){\makebox(0,0)[lb]{$\Sym\Lambda_{\Q}/\langle \varpi_i\alpha_i\mid i\in I \rangle$}}%
%
\special{pn 8}%
\special{pa 2270 2000}%
\special{pa 1950 1680}%
\special{fp}%
\special{sh 1}%
\special{pa 1950 1680}%
\special{pa 1983 1741}%
\special{pa 1988 1718}%
\special{pa 2011 1713}%
\special{pa 1950 1680}%
\special{fp}%
%
\special{pn 8}%
\special{pa 2960 2000}%
\special{pa 3280 1680}%
\special{fp}%
\special{sh 1}%
\special{pa 3280 1680}%
\special{pa 3219 1713}%
\special{pa 3242 1718}%
\special{pa 3247 1741}%
\special{pa 3280 1680}%
\special{fp}%
\put(26.1000,-14.5000){\makebox(0,0)[lb]{{\scriptsize $\Psi^*$}}}%
\put(16.9000,-19.2000){\makebox(0,0)[lb]{{\scriptsize \eqref{eq: coh generators 4}}}}%
\put(32.0000,-19.2000){\makebox(0,0)[lb]{{\scriptsize \eqref{eq: coh generators 3}}}}%
\end{picture}}%
\]

\vspace{10pt}

\begin{remark}
{\rm 
Harada-Horiguchi-Masuda's presentation \eqref{eq: coh generators 3} implies that we have $c_1(\LB_{\varpi_i})c_1(\LB_{\alpha_i})=0$ for $i\in I$ in the cohomology ring $H^*(Y;\Q)$.
In fact, 
by Proposition~\ref{prop: empty intersection}, the pair $(\psi_{\varpi_i},\phi_{\alpha_i})$ is a nowhere-zero section of the direct sum bundle $\LB_{\varpi_i}\oplus \LB_{\alpha_i}$ over $Y$. Therefore, the Euler class $e(\LB_{\varpi_i}\oplus \LB_{\alpha_i})(=c_1(\LB_{\varpi_i})c_1(\LB_{\alpha_i}))$ must vanish in $H^*(Y;\Q)$ (\cite[Property 9.7]{Milnor-Stasheff}).
This gives a simple geometric explanation for the equality $c_1(\LB_{\varpi_i})c_1(\LB_{\alpha_i})=0$.
}
\end{remark}

\vspace{5pt}

\begin{remark}
{\rm 
The cohomology ring $H^*(Y;\Q)$ is also isomorphic to $W$-invariant subring of the cohomology ring of the corresponding regular semisimple Hessenberg variety (see Remark~\ref{rem: generalization}), where the latter variety is in fact the non-singular toric variety associated with the collection of Weyl chambers. An explicit presentation of this $W$-invariant subring is known from the result of Klyachko \cite{Klyachko} which agrees with that given above.
}
\end{remark}

\vspace{20pt}

\section{Appendix}
We give a proof of Lemma~\ref{lem: intersection of two cones} here. Recall that 
\begin{align*}
 \ctc{J}{K} = \text{cone}( \{-\alpha^{\vee}_j \mid j\in J\}\cup\{\cvarpi_k \mid k\in K \} )\subset \mathfrak{t}_{\R}=\CWL\otimes_{\Z}\R
\end{align*}
for $K,J\subseteq I$ with $K\cap J=\emptyset$ 
and that $\fan=\{\ctc{J}{K} \mid J,K\subseteq I, \ J\cap K=\emptyset\}$.\\

\noindent
\textbf{Lemma~A.}\!\!
\textit{
For $\ctc{J}{K}, \ctc{P}{Q}\in \fan$, we have
\begin{align*}
 \ctc{J}{K}\cap \ctc{P}{Q} = \ctc{J\cap P}{K\cap Q}.
\end{align*}
In particular, the intersection of two cones in $\fan$ is a face of each.
}

\begin{proof}
Since it is obvious that $\ctc{J}{K}\cap \ctc{P}{Q} \supseteq \ctc{J\cap P}{K\cap Q}$ by definition, we prove that the opposite inclusion holds.
For that purpose, take an arbitrary element $v\in \ctc{J}{K}\cap \ctc{P}{Q}$.
Then $v$ can be written in two ways:
\begin{align}\label{eq: nonnegative expansion in cone}
 v = 
 \sum_{j\in J} x_j (-\alpha^{\vee}_j) + \sum_{k\in K} x_k \cvarpi_k 
 =\sum_{p\in P} y_p (-\alpha^{\vee}_p) + \sum_{q\in Q} y_q \cvarpi_q
\end{align}
for some coefficients $x_j,x_k\in \R_{\ge0}$ $(j\in J,\ k\in K)$ and $y_p,y_q\in \R_{\ge0}$ $(p\in P,\ q\in Q)$.
We first prove that 
\begin{align}\label{eq: vanishing for J-P}
 x_j=0 \qquad (j\in J- P)
\end{align}
in what follows.
By taking the pairing $\langle -\alpha_i,\ \ \rangle$ to $v$ in \eqref{eq: nonnegative expansion in cone} for $i\in J$, we obtain that
\begin{align*}
 \sum_{j\in J} \langle \alpha_i,\alpha^{\vee}_j \rangle x_j
 = \sum_{p\in P} \langle \alpha_i,\alpha^{\vee}_p \rangle y_p - \sum_{q\in Q} \delta_{iq} y_q
 \qquad (i\in J)
\end{align*}
since $\langle \alpha_i,\cvarpi_k \rangle=\delta_{ik}$ and $J\cap K=\emptyset$. 
We regard $J(\subseteq I)$ as a Dynkin diagram as in the proof of Lemma~\ref{lem: dim of cones}.
Then the coefficient matrix $C_J\coloneqq(\langle \alpha_i,\alpha^{\vee}_j \rangle)_{i,j\in J}$ appearing in this equation is the Cartan matrix associated with the Dynkin diagram $J$, and hence it is invertible over $\Q$.
Thus, for $j\in J$, we obtain 
\begin{align*}
 x_{j}
 &= \sum_{i\in J} (C_J^{-1})_{j,i} \left( \sum_{p\in P} \langle \alpha_i,\alpha^{\vee}_p \rangle y_p - \sum_{q\in Q} \delta_{iq} y_q \right) \\
 &= \sum_{p\in P} \left( \sum_{i\in J} (C_J^{-1})_{j,i} \langle \alpha_i,\alpha^{\vee}_p \rangle \right) y_p - \sum_{i\in J} (C_J^{-1})_{j,i} \sum_{q\in Q} \delta_{iq} y_q .
\end{align*}
The first summand in the last expression can be written as 
\begin{align*}
 &\sum_{p\in P} \left( \sum_{i\in J} (C_J^{-1})_{j,i} \langle \alpha_i,\alpha^{\vee}_p \rangle \right) y_p\\
 &\hspace{30pt}=
 \sum_{p\in P\cap J} \left( \sum_{i\in J} (C_J^{-1})_{j,i} \langle \alpha_i,\alpha^{\vee}_p \rangle \right) y_p 
 +\sum_{p\in P- J} \left( \sum_{i\in J} (C_J^{-1})_{j,i} \langle \alpha_i,\alpha^{\vee}_p \rangle \right) y_p \\
 &\hspace{30pt}=
 \sum_{p\in P\cap J} \delta_{jp} y_p 
 +\sum_{p\in P- J} \left( \sum_{i\in J} (C_J^{-1})_{j,i} \langle \alpha_i,\alpha^{\vee}_p \rangle \right) y_p ,
\end{align*}
where the last equality follows since $\langle \alpha_i,\alpha^{\vee}_p \rangle=(C_J)_{i,p}$ for $i,p\in J$.
The first summand in this expression vanishes when $j\in J- P$. 
Therefore, we obtain
\begin{align}\label{eq: appendix 10}
 x_j
 &= 
 \sum_{p\in P- J} \left( \sum_{i\in J} (C_J^{-1})_{j,i} \langle \alpha_i,\alpha^{\vee}_p \rangle \right) y_p
 - \sum_{i\in J} (C_J^{-1})_{j,i} \sum_{q\in Q} \delta_{iq} y_q 
 \qquad (j\in J-P).
\end{align}
In this equality, we have $\langle \alpha_i,\alpha^{\vee}_p \rangle \le 0$  since it is the $(i,p)$-th component (with $i\ne p$) of the Cartan matrix associated with $I$.
Also, since $C_J$ is a Cartan matrix, its inverse $(C_J)^{-1}$ consists of non-negative rational numbers (\cite{lu-ti92} or \cite[Sect.~13.1]{Humphreys78}).
Thus, the right hand side of \eqref{eq: appendix 10} is not positive.
Since $x_j$ was taken as a non-negative real number in \eqref{eq: nonnegative expansion in cone}, it follows that $x_j=0$ for $j\in J-P$ as claimed in \eqref{eq: vanishing for J-P}.

By changing the role of $(J,K)$ and $(P,Q)$ in the above argument, we also obtain
\begin{align*}
 y_p=0 \qquad (p\in P- J).
\end{align*}
That is, \eqref{eq: nonnegative expansion in cone} is now written as
\begin{align}\label{eq: expanding in two ways 2}
 -\sum_{j\in J\cap P} x_j \alpha^{\vee}_j + \sum_{k\in K} x_k \cvarpi_k 
 = -\sum_{p\in J\cap P} y_p \alpha^{\vee}_p + \sum_{q\in Q} y_q \cvarpi_q .
\end{align}
By taking the pairing $\langle -\alpha_i,\ \ \rangle$ to this equality for $i\in J\cap P$, we obtain that
\begin{align*}
 \sum_{j\in J\cap P} \langle \alpha_i,\alpha^{\vee}_j \rangle x_j
 = \sum_{p\in J\cap P} \langle \alpha_i,\alpha^{\vee}_p \rangle y_p
 \qquad (i\in J\cap P)
\end{align*}
since $J\cap P$ has no intersections with both $K$ and $Q$.
In particular, we have
\begin{align*}
 \sum_{j\in J\cap P} \langle \alpha_i,\alpha^{\vee}_j \rangle (x_j - y_j) = 0
 \qquad (i\in J\cap P).
\end{align*}
The coefficient matrix $(\langle \alpha_i,\alpha^{\vee}_j \rangle)_{i,j\in J\cap P}$ in this equality is invertible since it is the Cartan matrix associated with the Dynkin diagram $J\cap P$. 
Thus, we obtain $x_j=y_j$ for $j\in J\cap P$.
This means that \eqref{eq: expanding in two ways 2} is now written as
\begin{align*}
 \sum_{k\in K} x_k \cvarpi_k 
 = \sum_{q\in Q} y_q \cvarpi_q .
\end{align*}
Since $\cvarpi_1,\ldots,\cvarpi_{\rkg}$ form a basis of $\mathfrak{t}_{\R}=\CWL\otimes_{\Z}\R$, this equality implies that
\begin{align}\label{eq: expanding in two ways 3}
 x_k = 0 \quad (k\in K-Q).
\end{align}
Now, by \eqref{eq: vanishing for J-P} and \eqref{eq: expanding in two ways 3}, it follows that
\begin{align*}
 v = \sum_{j\in J\cap P} x_j (-\alpha^{\vee}_j) + \sum_{k\in K\cap Q} x_k \cvarpi_k 
\end{align*}
which implies that $v\in \ctc{J\cap P}{K\cap Q}$. Thus we obtain $\ctc{J}{K}\cap \ctc{P}{Q} \subseteq \ctc{J\cap P}{K\cap Q}$, as desired.
\end{proof}

\vspace{10pt}

We next give an alternative proof for Kostant's result used in the proof of Proposition~\ref{prop: surjective Psi}.
For $i\in I$, we consider the function defined by
\begin{align}\label{eq: around identity 10''}
 q'_i \colon  Y\cap U^-B/B \rightarrow \C
 \quad ; \quad
 uB\ (u\in U^-) \mapsto -(\text{Ad}_{u^{-1}}\nil)_{-\alpha_i},
\end{align}
where $(\text{Ad}_{u^{-1}}\nil)_{-\alpha_i}$ is the coefficient of the root vector $\f{i}\in\mathfrak{g}_{-\alpha_i}$ for the root  decomposition of  
$\text{Ad}_{u^{-1}}\nil\in \mathfrak{g}$.\\

\noindent
\textbf{Lemma~B (\cite[Theorem~27]{Kostant}).}\!\!
\textit{
The functions $q'_1,\ldots,q'_{\rkg}$ in \eqref{eq: around identity 10''} are algebraically independent in the coordinate ring $\C[Y\cap U^-B/B]$, where we regard $Y\cap U^-B/B$ as a reduced scheme.
}

\begin{proof}
Let $D\coloneqq|\Phi^+|$ be the number of positive roots. 
For the following argument, we choose an enumeration of positive roots
\begin{align}\label{eq:enumeration}
\Phi^+ = \{\beta_1,\ldots,\beta_D\}.
\end{align}
Then, we have an isomorphism
\begin{align*}
\C^D \xrightarrow{\sim}  U^-B/B
\quad ; \quad 
(t_1,\ldots,t_D) \mapsto \exp(t_1F_{\beta_1}) \cdots \exp(t_D F_{\beta_D})B.
\end{align*}
In what follows, we identify the coordinate ring $\C[U^-B/B]$ and the polynomial ring $\C[t_1,\ldots,t_D]$ through this map with the grading 
\begin{align*}
\deg t_i = \text{ht} (\beta_i)\ge1
\qquad (1\le i\le D),
\end{align*}
where $\text{ht} (\beta_i)$ denotes the height of the positive root $\beta_i$.
For $\alpha\in\Phi^+$, let us denote 
\begin{align*}
f_{\alpha} \colon U^-B/B \rightarrow \C
\quad ; \quad 
uB\ (u\in U^-) \mapsto (\text{Ad}_{u^{-1}}\nil)_{-\alpha}.
\end{align*}
With the grading given above, the functions $f_{\alpha}$ for $\alpha\in\Phi^+$ are homogeneous polynomials of positive degrees.
In fact, for $u=\exp(t_1F_{\beta_1}) \cdots \exp(t_D F_{\beta_D})\in U^-$, we have
\begin{align*}
(\text{Ad}_{u^{-1}}\nil)_{-\alpha}
= \Big(\exp(\text{ad}(-t_D F_{\beta_D}))\cdots\exp(\text{ad}(-t_1 F_{\beta_1}))\nil\Big)_{-\alpha},
\end{align*}
and one can verify that this is a homogeneous polynomial in $\C[t_1,\ldots,t_D]$ (for the grading given above) by expanding the exponentials above (cf.\ the argument in the proof of \cite[Lemma~3.5]{ab-fu-ze}).

To prove the claim of this theorem, we first show that the sequence $f_{\beta_1},\ldots,f_{\beta_D}$ is a regular sequence in the coordinate ring $\C[U^-B/B]$.
For that purpose, observe that the zero locus
\begin{align*}
Z(f_{\beta_1},\ldots,f_{\beta_D})=
\{ uB \in U^-B/B \mid u\in U^-, \ (\text{Ad}_{u^{-1}}\nil)_{-\alpha}=0\ (\alpha\in\Phi^+)\}
\end{align*} 
is non-empty (since $eB$ belongs to this set) and that it 
is an open affine neighbourhood of the regular nilpotent Hessenberg variety associated with the smallest Hessenberg space $\mathfrak{b}$. 
This Hessenberg variety is $0$-dimensional (\cite[Corollary~3]{pre18} so that we obtain
\begin{align*}
\text{Krull dim } \C[U^-B/B]/(f_{\beta_1},\ldots,f_{\beta_D}) = 0. 
\end{align*} 
Since $f_{\beta_1},\ldots,f_{\beta_D}$ are homogeneous polynomials of positive degrees, 
this implies that the sequence $f_{\beta_1},\ldots,f_{\beta_D}$ is a regular sequence in $\C[U^-B/B]$ by \cite[the proof of Proposition~5.1]{fu-ha-ma}.

We next show that the sequence $q'_1$, \ldots, $q'_{\rkg}$ is a regular sequence in the reduced coordinate ring $\C[Y\cap U^-B/B]$.
For that purpose, we take the enumeration \eqref{eq:enumeration} so that the last $\rkg$ sequence in \eqref{eq:enumeration} is the sequence of the simple roots $\alpha_1,\ldots,\alpha_{\rkg}$.
Then it follows from \cite[Corollary~3.7]{ab-fu-ze} that the reduced coordinate ring $\C[Y\cap U^-B/B]$ is given by the quotient
\begin{align}\label{eq: app2 30}
 \C[Y\cap U^-B/B] = \C[U^-B/B]/(f_{\beta_1},\ldots,f_{\beta_{D-\rkg}}),
\end{align}
where we do not list the last $\rkg$ sequence $f_{\beta_{D+1-\rkg}}(=f_{\alpha_1}),\ldots,f_{\beta_{D}}(=f_{\alpha_{\rkg}})$ in the ideal.
Let us denote by $\overline{f_{\alpha_i}}$ the image of $f_{\alpha_i}$ in $\C[Y\cap U^-B/B]$.
Then we have $q'_i=-\overline{f_{\alpha_i}}$ for $1\le i\le \rkg$.
Since we know that the sequence $f_{\beta_1},\ldots,f_{\beta_D}$ is a regular sequence in $\C[U^-B/B]$, it follows from the definition of regular sequences that the last $\rkg$ sequence $\overline{f_{\alpha_1}},\ldots,\overline{f_{\alpha_{\rkg}}}$ is a regular sequence in $\C[Y\cap U^-B/B]$ (see \eqref{eq: app2 30}). 
Since an arbitrary regular sequence in an arbitrary $\C$-algebra consists of algebraically independent elements, the claim follows.
\end{proof}

\vspace{20pt}

Lastly, we give a proof of \eqref{eq: pushing forward zero schemes} in Section~\ref{subsec: corresponding on degree 2}. For future use, we prove it in a general setting.

Let $X$ be a variety and $\mathcal{Z}$ a closed subscheme of $X$ with the underlying closed subset $Z\subseteq X$ which is equidimensional, i.e., the irreducible components $Z_1,\ldots,Z_r$ of $Z$ have the same dimension.
For example, the zero scheme associated to an effective Cartier divisor on $X$ always satisfies this condition.
Writing $d\coloneqq \dim \mathcal{Z}$, the scheme $\mathcal{Z}$ determines a cycle $\cycle{\mathcal{Z}}$ in $H_{2d}(X;\Q)$; it is a linear combination of the cycles $\cycle{Z_1},\ldots,\cycle{Z_r}$ with positive integer coefficients given by the \textit{geometric multiplicities} (see \cite[Sect.~1.5 and Sect.~19.1]{Ful98}). 
Namely, we have
\begin{align}\label{eq: schemey cycle} 
 \cycle{\mathcal{Z}} 
 = a_1\cycle{Z_1} + \cdots + a_r\cycle{Z_r}
\end{align} 
with $a_i=\ell_{\mathscr{O}_{Z_i,\mathcal{Z}}}(\mathscr{O}_{Z_i,\mathcal{Z}})\in\Z_{\ge1}$ the length of the local ring $\mathscr{O}_{Z_i,\mathcal{Z}}$
$(1\le i\le r)$.
Suppose that we have a \textit{proper} morphism $f\colon X\rightarrow Y$ to a variety $Y$, and suppose that $W$ is an irreducible closed subset of $Y$ such that $f(Z)\subseteq W$.
Under this setting, the following hold.\\

\noindent
\textbf{Lemma~C.}\!\!
\textit{
If $\dim Z=\dim W$, 
there exists a non-negative integer $\ell\in\Z_{\ge0}$ such that
\begin{align*}
 f_*(\cycle{\mathcal{Z}}) = \ell \cycle{W}
 \quad \text{in $H_{2d}(Y;\Q)$.}
\end{align*} 
}

\vspace{-10pt}

\begin{proof}
By \eqref{eq: schemey cycle}, we have
\begin{align}\label{eq: app C 10}
 f_*(\cycle{\mathcal{Z}}) = a_1 f_*(\cycle{Z_1}) + \cdots + a_r f_*(\cycle{Z_r}).
\end{align} 
For each $f_*(\cycle{Z_i})$ in the right hand side, it follows from \cite[Lemma 19.1.2]{Ful98} that there exists $\ell_i\in\Z_{>0}$ such that 
\begin{align*}
 f_*(\cycle{Z_i}) = 
 \begin{cases}
  \ell_i\cycle{f(Z_i)} \quad &\text{if $\dim f(Z_i)=\dim Z_i(=d)$}, \\
  0 &\text{otherwise}.
 \end{cases}
\end{align*}  
Since $f$ is proper, $f(Z_i)$ is an irreducible closed subset of $W$, and we have $\dim W=d$ by the assumption.
Hence, if $\dim f(Z_i)=d$, then we have $f(Z_i)=W$.
Thus, the previous equality can be expressed as
\begin{align*}
 f_*(\cycle{Z_i}) = 
 \begin{cases}
  \ell_i\cycle{W} \quad &\text{if $\dim f(Z_i)=\dim Z_i(=d)$}, \\
  0 &\text{otherwise}.
 \end{cases}
\end{align*} 
Now we obtain the desired claim from \eqref{eq: app C 10} since we have $a_i=\ell_{\mathscr{O}_{Z_i,\mathcal{Z}}}(\mathscr{O}_{Z_i,\mathcal{Z}})\ge1$ for $1\le i\le r$.
\end{proof}


\end{document}